\documentclass[11pt]{article}

\usepackage[T1]{fontenc}
\usepackage[margin=1in]{geometry} % wide, nice margins for arXiv

\usepackage{graphicx}
\usepackage{float}

\usepackage{mathtools} % loads amsmath
\usepackage{amssymb,amsfonts}
\usepackage{bm}
\usepackage{xcolor}
\usepackage[normalem]{ulem} % for \sout etc.

\usepackage{microtype} % optional but recommended

\usepackage{amsthm}

\numberwithin{equation}{section}

\theoremstyle{plain}
\newtheorem{theorem}{Theorem}[section]
\newtheorem{lemma}[theorem]{Lemma}
\newtheorem{proposition}[theorem]{Proposition}

\theoremstyle{definition}

\newtheorem{example}[theorem]{Example}

\theoremstyle{remark}

\newtheorem{myclaim}[theorem]{Claim} % your custom claim env

\usepackage{hyperref}
\usepackage{cleveref}
\usepackage{enumitem}
\hypersetup{
	colorlinks,
	linkcolor={red!50!black},
	citecolor={blue!50!black},
	urlcolor={blue!80!black}
}
\newcommand{\ones}{\bm{1}}
\newcommand{\zeros}{\bm{0}}
\newcommand{\N}{{\mathbb{N}}}
\newcommand{\Z}{{\mathbb{Z}}}

\newcommand{\R}{\mathbb{R}}
\newcommand{\Sym}{\mathcal{S}}

\newcommand{\fix}{\normalfont\text{fix}}
\newcommand{\orb}{\text{orb}}
\newcommand{\LS}{\normalfont\text{LS}}
\newcommand{\SA}{\normalfont\text{SA}}
\newcommand{\conv}{\text{conv}}
\newcommand{\M}{{\normalfont\text{M}}}

\title{Linear Programming Hierarchies Collapse under Symmetry}

\author{
  Yuri Faenza\thanks{Department of Industrial Engineering and Operations Research, Columbia University, USA}
  \and 
  Victor Verdugo\thanks{Department of Industrial and Systems Engineering, Pontificia Universidad Católica de Chile, Chile}
  \thanks{Institute for Mathematical and Computatonal Engineering, Pontificia Universidad Católica de Chile, Chile}
  \and 
  José Verschae\footnotemark[3]
  \and 
  Matías Villagra\footnotemark[1]
}

\date{} % usually no date for arXiv

\begin{document}

\maketitle

\begin{abstract}
The presence of symmetries is one of the central structural features that make some integer programs challenging for state-of-the-art solvers. In this work, we study the efficacy of Linear Programming (LP) hierarchies in the presence of symmetries. 
Our main theorem unveils a connection between the algebraic structure of these relaxations and the geometry of the initial integer-empty polytope: We show that under $(k+1)$-transitive symmetries--a measure of the underlying symmetry in the problem--the corresponding relaxation at level $k$ of the hierarchy is non-empty if and only if the initial polytope intersects all $(n-k)$-dimensional faces of the hypercube. 
In particular, the hierarchies of Sherali-Adams, Lovász-Schrijver, and the Lift-and-Project closure are equally effective at detecting integer emptiness. 
Our result provides a unifying, group-theoretic characterization of the poor performance of LP-based hierarchies, and offers a simple procedure for proving lower bounds on the integrality gaps of symmetric polytopes under these hierarchies. 
\end{abstract}

\section{Introduction}

Understanding what makes an Integer Program (IP) challenging for state-of-the-art algorithms is a fundamental step toward developing more effective solvers. 
While it is reasonable to assume that a larger number of variables and constraints would slow down any algorithm for solving an IP, there are other, more subtle features of IPs that are problematic due to the techniques we currently employ. 
One such feature is \emph{symmetry}, informally defined as any transformation of a mathematical object that leaves its fundamental properties unchanged.

Mathematicians usually welcome a high degree of symmetry, as it often implies strong properties of the objects under consideration. 
For instance, systems of linear equalities $Ax=b$ can be solved faster in practice when $A=A^{\top}$, because factorizations of $A$ are guaranteed to exist in this case, see, e.g.,~\cite{bunch_stable_1977}. 
In contrast, symmetry is usually problematic in Integer Programming. 
Even small-sized symmetric problems are often hard to solve~\cite{fulkerson2009two,margot2010,ostrowski2011solving,ostrowski_using_2014,pfetsch_computational_2019}. 
In his survey on the topic~\cite{margot2010}, Margot states that \emph{the
trouble comes from the fact that many subproblems in the enumeration tree are isomorphic, forcing a wasteful duplication of effort}--but as we will see, this is not the only problem caused by symmetry. 

Understanding how to tackle symmetric IPs is essential, as some degree of symmetry is prevalent in both theory and practice. 
The structure of the problem itself sometimes causes symmetry -- for instance, if we wish to schedule jobs on identical machines, or if we solve a combinatorial optimization problem on a highly symmetric graph. 

Whichever the cause, more than 36\%  of the instances from MIPLIB 2010 admit nontrivial symmetries; see the work by Pfetsch and Rehn~\cite{pfetsch_computational_2019} for details. Moreover, symmetries have played an important role for theoretical work, both from a polyhedral perspective (e.g.,~\cite{yannakakis_expressing_1991,ben-tal_polyhedral_2001}), as a devise to construct lower bounds of relaxations (e.g.,~\cite{grigoriev2001complexity,laurent_lower_2003}), and a way of reducing the dimensionality of convex relaxations (e.g.,~\cite{gatermann_symmetry_2004,ostrowski_using_2014}), to mention a few.%

\paragraph{Our contributions.} In this work, we investigate how symmetry can harm state-of-the-art techniques for solving Integer Programs. 
Our work distinguishes itself from the literature with respect to two features. 
First, unlike the most common approach of studying the effect of symmetry on creating many ``equivalent'' integer solutions, we investigate the relationship between symmetry and the performance of {\it linear programming hierarchies}, which are well-known techniques that iteratively improve LP relaxations of an integer program. 
Second, while most research has focused on specific classes of problems such as, e.g., matchings~\cite{kurpisz2018sum,mathieu2009sherali}, knapsack~\cite{grigoriev2001complexity,karlin2011integrality}, max-cut and sparsest-cut~\cite{cohen20242,gupta2013sparsest,laurent_lower_2003}, scheduling~\cite{kurpisz_semidefinite_2016,verdugo_breaking_2020}, vertex-cover~\cite{georgiou2010integrality}, and planted clique~\cite{barak2019nearly,meka_sum--squares_2015}, with the notable exception of the work by Kurpisz et al.~\cite{kurpisz_sum--squares_2020}, our assumptions only rely on a well-established group-theoretic parametrization of the symmetry of the underlying polytope, namely \emph{$k$-transitivity}. 
Our main result is the following.

\begin{theorem}\label{main_theorem}
    Let $k, n \in \N$ such that $k < n$, and let $P \subseteq [0,1]^{n}$ be a $(k+1)$-transitive integer-empty polytope. Then the following statements are equivalent:
    \begin{enumerate}[label=\normalfont(\Alph*)]
    \setlength\itemsep{0.3em}
        \item $P$ intersects all of the $(n-k)$-dimensional faces of the hypercube $[0,1]^{n}$;\label{mainthm-a}
        \item $\SA^{k}(P)\neq \emptyset$;\label{mainthm-b}
        \item $\LS^{k}(P) \neq \emptyset$;\label{mainthm-c}
        \item $\LS_{0}^{k}(P) \neq \emptyset$.\label{mainthm-d}
    \end{enumerate}
\end{theorem}

We defer to Section~\ref{section:preliminaries} the definition of the hierarchies $\SA^k$, $\LS^k$, and $\LS_0^k$. 
A polytope $P$ is $G$-invariant for a permutation group $G$ over $[n]:=\{1,\ldots,n\}$, if under the action of the elements of $G$, given by $\pi x:= (x_{\pi^{-1}(1)},\ldots,x_{\pi^{-1}(n)})$, we have that $P=\pi(P)$ for all $\pi\in G$. 
A group $G$ is $k$-transitive if it can map any ordered sequence of $k$ indices in $[n]$ to any other ordered sequence of $k$ indices in $[n]$ by some element in $G$. 
A polytope is $k$-transitive if it is $G$-invariant for a $k$-transitive group $G$; 
observe that $(k+1)$-transitivity implies $k$-transitivity. 
Thus, the larger the $k$, the {\it more symmetric} $P$ is. 

\smallskip 

A high-level description of Theorem~\ref{main_theorem} is as follows. Assume that the polytope $P\subseteq [0,1]^n$ is integer-empty and ``symmetric enough'', i.e., $(k+1)$-transitive. 
Theorem~\ref{main_theorem} implies that we can decide whether $k$ rounds of any of the three hierarchies certify that $P$ is integer-empty just by looking at the intersection of $P$ with the faces of the hypercube $[0,1]^n$. In particular, one of the hierarchies gives us such a certificate if and only if all of them do. 

It is well-known that, for any $k \in \mathbb{N}$ and polytope $P\subseteq [0,1]^n$,  $\SA^k(P)\subseteq \LS^k(P) \subseteq \LS_0^k(P)$, and in general they can be significantly different. In contrast, in our setting, we observe a {\it hierarchy collapse}: the three methods are equally effective in determining whether the polytope is integer-empty. 
This result may explain why the integer-emptiness of symmetric polytopes is hard to prove for LP-based hierarchies: the strongest of the three operators does not outperform the weakest one. Moreover, if the other features stay constant, a ``more symmetric'' polytope (i.e., $k$-transitive for larger $k$) will ``resist'' these operators for at least as many rounds as a ``less symmetric'' polytope.

Although the concept of $k$-transitivity is well-known in group theory--see, e.g.,~\cite[Sec. 7]{dixon+mortimer96}--we believe that using it as a way of measuring the level of symmetry of a polytope--especially in the context of Integer Programming--is one of the contributions of this work. We briefly discuss in Example~\ref{ex:STS} how our approach can be employed to obtain bounds on the number of rounds the hierarchies need to improve the bound given by an LP relaxation. We further discuss in Appendix~\ref{appendix:hypothesis-thm} how relaxing some of the hypotheses of Theorem~\ref{main_theorem} makes the statement false.

\smallskip

\noindent \emph{Organization of the paper.} In Section~\ref{section:applications}, we discuss some applications and consequences of Theorem~\ref{main_theorem}. 
In Section~\ref{section:preliminaries} we introduce the operators that will be used throughout the paper, along with definitions and results on permutation groups that are needed for Theorem~\ref{main_theorem}. In Section~\ref{section:AimpliesB} we prove Theorem~\ref{main_theorem}, after introducing additional tools from polyhedral and group theory. 
The implication (A)$\Rightarrow$(B) of Theorem~\ref{main_theorem} is the technical core of this work, thus it is the one we dedicate most space to. Section~\ref{section:conclusion} concludes with a discussion and open questions.

\smallskip

\noindent \emph{Further related work.}\label{sec:related-work}
Linear and semidefinite hierarchies offer systematic procedures for strengthening an initial relaxation of an integer program~\cite{lasserre2001global,lovasz+schrijver91,sherali1990hierarchy,sherali+adams90}; see, e.g.,~\cite{laurent03} for a detailed comparison.
They are used as an automatic way to generate stronger relaxations, while keeping efficiency at low levels of the hierarchy.
However, their success in improving integrality gaps is limited, and symmetries partially explain this phenomenon.
Machine scheduling showcases this behavior: while Sherali-Adams and Lasserre are not able to close the gap after linearly many rounds over the symmetric assignment relaxation, they yield approximation schemes after symmetry-breaking~\cite{kurpisz_semidefinite_2016,verdugo_breaking_2020}. 
Kurpisz et al.~\cite{kurpisz_sum--squares_2020} study the impact of symmetries for Lasserre and show a sufficient condition to prove rank lower bounds whenever the solutions and constraints are invariant under every permutation of $[n]$.

Several approaches for handling symmetries in IPs are found in the literature, including symmetry-breaking procedures in the formulation or in the algorithms employed; see, e.g.,~\cite{faenza2009extended,kaibel2011orbitopal,kaibel2008packing} and the survey by Pfetsch and Rehn for a detailed exposition~\cite{pfetsch_computational_2019}.
Towards the goal of providing a unified perspective, van Doornmalen and Hojny~\cite{van2025unified} recently introduced a framework that captures many of the existing symmetry-handling methods. 
In the context of LP hierarchies, Ostrowski~\cite{ostrowski_using_2014} explored the impact of isomorphism pruning for symmetric Sherali-Adams relaxations.

\section{Examples and applications}\label{section:applications}

\begin{example}{($3$-dimensional example).}\label{ex:3D} We first illustrate our theorem on the case of a 3-dimensional polytope. 
Let $\text{conv}(S)$ denote the convex hull of a set $S\subseteq \mathbb{R}^n$. 
In Figure~\ref{fig:3dimExample} we consider the polytope $P=\text{conv}(\{\pi(0,1/2,1): \pi\in \mathcal{S}_3\})$, that is, the convex hull of all vectors obtained by permuting the coordinates of $(0,1/2,1)$ in every possible way. 
It is easy to see that $P$ is $3$-transitive by construction (and hence $2$-transitive). In the left figure, one can see the original polytope. 
Observe that $P$ intersects every $2$-dimensional face of $[0,1]^3$. On the right, we depict $\SA^1(P)$, $\LS^1(P)$, and $\LS_0^1(P)$. As predicted by Theorem~\ref{main_theorem}, when taking $k=1$ these three polytopes are non-empty; however they are not necessarily equal as $\SA^1(P)=\LS^1(P)\subsetneq\LS_0^1(P)$. On the other hand, since $P$ does not intersect every $1$-dimensional face (e.g., the face defined by $x_1=1$ and $x_2=1$ does not intersect $P$), Theorem~\ref{main_theorem} implies that $\SA^2(P)=\LS^1(P)=\LS_0^1(P)=\emptyset$.
\end{example}

\begin{figure}
\begin{center}
\includegraphics[width=0.35\textwidth]{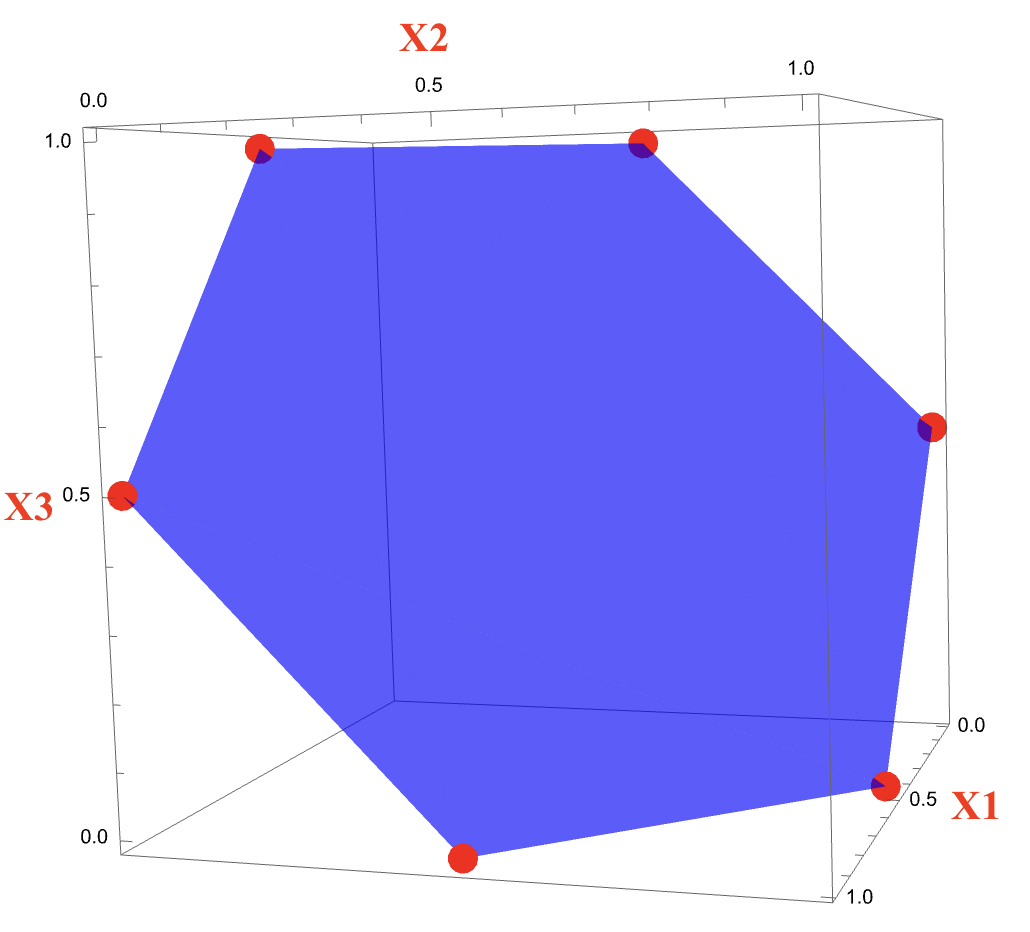}
\hspace{2cm}
\includegraphics[width=0.35\textwidth]{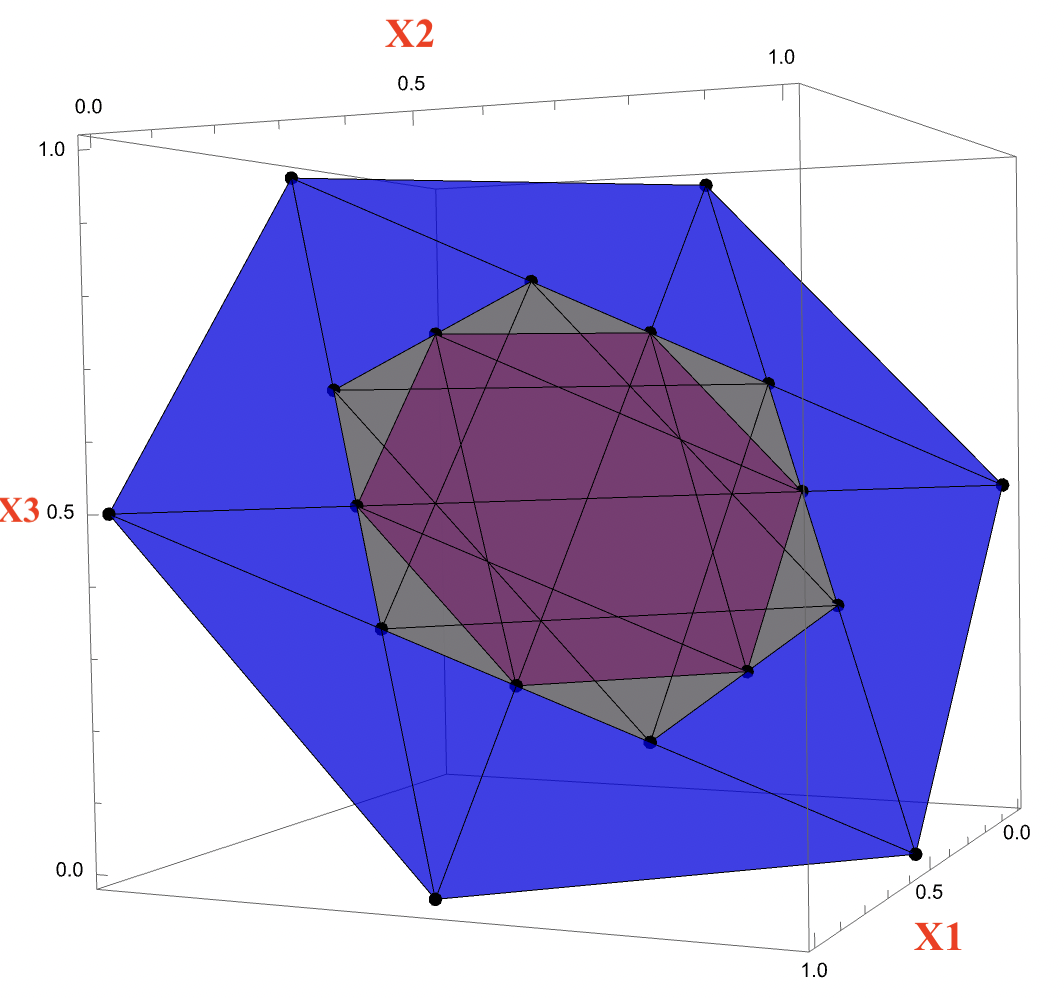}
\end{center}
\caption{A depiction of $P=\text{conv}(\{\pi(0,1/2,1): \pi\in \mathcal{S}_3\})$ (left), $\SA^1(P)=\LS^1(P)$ (right, violet), and $\LS_0^1(P)$ (right, gray).}
\label{fig:3dimExample}
\end{figure}

\begin{example}{(Steiner triple cover and LP bounds).}\label{ex:STS} Starting with the work of Fulkerson, Nemhauser, and Trotter~\cite{fulkerson2009two}, Steiner Triple Covering Problems (STSs) have been studied computationally by many researchers in Integer Programming~\cite{feo1989probabilistic,mannino1995solving,ostrowski_using_2014,ostrowski2011solving}. Most of the work in the area has focused on two classes of STSs, whose number of variables is a power of $3$, or $5$ times a power of $3$. For $n \in \N$, we let STS$_n$ be STSs instances with $3^n$ variables, and give a formal description of such instances in Appendix~\ref{app:example:steiner}. For $n \in \N_{\geq 2}$, the feasible region of the LP relaxation of STS$_n$ is known to be $2$-transitive~\cite{hall1960automorphisms} (it is also easy to argue that it is not $3$-transitive, see again Appendix~\ref{app:example:steiner} for details). One can add a (non-optimal) objective function cut so that the resulting optimal face of the Linear Programming relaxation satisfies the hypothesis of Theorem~\ref{main_theorem} with $k=1$, and property (A) (see again Appendix~\ref{app:example:steiner}). Thus, one round of any of the three hierarchies considered in this paper does not suffice to improve the bound given by the LP relaxation where the cut has been added.
STS$_n$ illustrates how Theorem~\ref{main_theorem} can be employed to give a bound on the number of rounds we need to apply the hierarchies for, in order to improve the bound given by the LP relaxation of a (symmetric) integer program.
\end{example}

\begin{example}{(Parity and Knapsack-Cover).}
Understanding the strength of hierarchies to prove approximate bounds has attracted significant attention. In particular, constructing lower bounds is a technically challenging but important task in order to understand the limitation of such techniques. The Sum-of-Squares (SoS) hierarchy, a stronger construction than $\SA$, $\LS$, and $\LS_0$, has been studied extensively in this setting. The starting point of several of these constructions~\cite{laurent_lower_2003,meka_sum--squares_2015,kurpisz_sum--squares_2016} was given by Grigoriev~\cite{grigoriev2001complexity}, who showed that $k=\lfloor n/2 \rfloor$ rounds of the SoS hierarchy are needed to prove that $P=\{x\in [0,1]^n: \sum_{i=1}^n x_i =n/2\}$ is integer-empty when $n$ is odd. Noting that $P$ is $\mathcal{S}_n$-invariant, and thus $n$-transitive, our main theorem implies directly that the (weaker) Sherali-Adams hierarchy needs exactly $k+1=\lfloor n/2 \rfloor+1$ rounds to prove that $P$ is integer-empty. Although this is not a new result, it gives a very simple proof for this fact and exemplifies the usefulness of our techniques: note that it suffices to find a vector in the intersection of $P$ with each $(n-k)$-dimensional face of $[0,1]^n$, a straightforward exercise for this polytope. Moreover, we note that the construction of a feasible point of $\SA^k(P)$ we use in the proof of (A)$\Rightarrow$(B) in Theorem~\ref{main_theorem} is a strict generalization of the ideas of Grigoriev~\cite{grigoriev2001complexity} to full-dimensional polytopes. 

Another application considers the Knapsack-Cover problem (also known as Min-Knapsack). While the basic linear relaxation of this problem has an unbounded integrality gap, the introduction of knapsack-cover inequalities reduces this gap to 2. Despite efforts to tighten this relaxation using hierarchies like Sherali-Adams, Lovász-Schrijver, and SoS, the integrality gap remains at 2 even after several rounds. We show that our construction generalizes the approach by Bienstock and Zuckerberg~\cite{daniel_bienstock_approximate_2003,bienstock_subset_2004}, offering a simpler method to generate feasible vectors in the Sherali-Adams hierarchy. 
For a specific instance, the knapsack-cover linear relaxation yields $$P=\textstyle\min\left\{\sum_{i=1}^n x_i: \sum_{i=1}^n x_i\ge 1+\frac{1}{n-1}, \sum_{j\in [n]\setminus\{i\}}x_j\ge 1,x\in [0,1]^n\right\}.$$
In Appendix~\ref{app:KCover}, we show that our techniques of Section~\ref{section:AimpliesB} gives a direct method for constructing a solution in $x\in \SA^k(P)$ with value $\sum_{i=1}^n x_i\le 1+O(k/n)$, while the optimal integral solution has value 2, which implies an integrality gap of $2$ as long as $k$ is sublinear, $k=o(n)$. 

We emphasize that in both constructions our explicit solution is guaranteed to belong to the Sherali-Adams hierarchy; we leave as an open problem if the same solution belongs to the tighter SoS hierarchy at the same level.
\end{example}

\begin{example}{(Cropped cube).}
The cropped cubed is a polytope obtained by cropping each vertex of the 0-1 hypercube,  given by $$P^{cc} := \textstyle\left\{ x \in [0,1]^{n} \, : \, \sum_{i \in I} x_{i} + \sum_{i \notin I} (1-x_{i}) \geq 1/2, \, \text{ for all } I \subseteq [n] \right\}.$$
This polytope is integer-empty and $n$-transitive. Goemans and Tun\c{c}el~\cite{goemans+tuncel01} showed that it resists $n-1$ rounds of the Lovász-Schrijver operator, meaning that $\LS^{n-1}(P^{cc})\neq \emptyset$. Later, Laurent~\cite{laurent_lower_2003} extended this result, showing that even $\SA^{n-1}(P^{cc})\neq \emptyset$. Theorem~\ref{main_theorem} implies directly these results: we only need to observe that $P^{cc}$ intersects every $n-1$ dimensional face of $[0,1]^n$. Namely, we need to verify that for any two disjoint sets $S_0,S_1\subseteq[n]$ with $|S_0\cup S_1|=n-1$, there is a vector $x\in P^{cc}$ with $x_i=0$ if $i\in S_0$ and $x_i=1$ for $i\in S_1$. To do so, simply define the only free variable as $x_i=1/2$, for $i\in[n]\setminus(S_1\cup S_2)$.
\end{example}

\section{Preliminaries}\label{section:preliminaries}

For $n \in \N$ we denote $\{1,\dots,n\}$ by $[n]$ and $[0] := \emptyset$. For $i,j \in \Z_{\geq 0}$ and $x \in \R^{n}$ we write $\binom{i}{j} := 0$ and $\prod_{j}^{i} x_{j} := 1$ if $j > i$. Moreover, we define $\binom{i}{r} := 0$ and $\prod_j^r x_j:=1$ if $r \in \Z_{<0}$. We denote by $\ones$ and $\zeros$ the all-ones vector and the all-zeros vector, respectively. Moreover, we write $\binom{n}{\leq i} := \sum_{j=0}^{i} \binom{n}{j}$. For $\emptyset \neq S \subseteq [n]$ and subsets $J_{0},J_{1} \subseteq S$, we say that $(J_{0},J_{1})$ is a $2$-partition of $S$ if $J_{0} \cup J_{1} = S$ and $J_{0} \cap J_{1} = \emptyset$, and we denote it as $J_{0} \sqcup J_{1}  = S$. For any vectors $y \in \R^{\binom{n}{\leq j}}$ and $z \in \R^{\binom{n}{\leq i}}$, let $(y,z)$ denote their vertical concatenation, i.e., $(y,z) := \bigl(\!\begin{smallmatrix} y \\ z \end{smallmatrix}\!\bigr)$. Whenever $I = \{j\}$, we simply write $y_I=y_{j}$. Throughout this section, we let $m \in \mathbb{N}$ and $P := \{ x \in \R^{n} \, : \, a_{i}^\top x \geq b_{i}, i \in [m]\} \subseteq [0,1]^{n}$. For technical reasons, we will assume that the description of $P$ explicitly contains the linear inequalities $\zeros \le x\le \ones$.

\subsection{Linear Programming Hierarchies} 

We now present the Linear Programming hierarchies of interest for this work. 
For a polytope $P$ as above, let
\begin{equation*}
    P^{j,1} := P \cap \{ x \in \R^{n} \, : \, x_{i} = 1\} \quad \text{and}\quad P^{j,0} := P \cap \{ x \in \R^{n} \, : \, x_{i} = 0\} \text{ for $j \in [n]$.}
\end{equation*}
 We denote by $P_{i}$ the \emph{Balas-C\'eria-Cornu\'ejols operator}~\cite{balas1993lift} by the variable $x_{i}$, i.e., $P_{i}(P) := \conv( P^{i,1} \cup P^{i,0})$. Hence the \emph{Lift-and-Project closure} of $P$ can be written as $\LS_0(P):=\bigcap_{i\in[n]} P_{i}(P)$. We will employ a useful characterization of the Lift-and-Project closure (see, e.g.,~\cite{lovasz+schrijver91,goemans+tuncel01}): 
\begin{align*}
    \LS_{0}(P) := \{ x \in \R^{n} \, : \, &\exists Y \in \R^{(n+1) \times (n+1)}, \, Ye_{i}, Y(e_{0} - e_{i}) \in K(P), \forall i \in [n],\\
    &Ye_{0} = Y^{\top} e_{0} = \text{diag}(Y) = (1, x)^{\top} \},
\end{align*}
where the matrices $Y$ have columns and rows indexed by $\{0,1,\ldots,n\}$, $e_i$ denotes the $i$-th canonical vector, and $K(P) := \{ (\lambda, \lambda x)^{\top}:\lambda \in \R_{+}, \, x \in P\}$. 
We denote by $\LS_{0}^{k}$ the $k$ iterative applications of $\LS_{0}$. 
That is, $\LS_{0}^{2}(P)=\LS_{0}(\LS_{0}(P))$, $\LS_{0}^{3}(P)=\LS_{0}(\LS_{0}^2(P))$, etc. The \emph{Lov\'asz-Schrijver} operator, $\LS$, further requires the lifting matrices $Y$ to be symmetric, and we similarly denote repeated applications of $\LS$ by $\LS^{k}$. 
Following~\cite{fleming+etal19}, we define the \emph{Sherali-Adams level $k$} relaxation $\M^{k}(P)$ of $P$~as follows:
\begin{enumerate}[label=(\roman*)]
    \item \textbf{Extend}: For every row $a_{i}^\top x \geq b_{i}$, with $i \in [m]$, every set $S \subseteq [n]$, $|S|=k$ and every $2$-partition $J_{0} \sqcup J_{1} = S$, introduce a new constraint
    \begin{equation}\label{definition:SA_ineq}
         \prod_{i \in J_{1}} x_{i} \, \prod_{j \in J_{0}} (1-x_{j}) \cdot (a_{i}^\top x - b_{i} \geq 0).
    \end{equation}
    \item \textbf{Linearize}: multilinearize each of the constraints in the previous step by replacing every monomial $\prod_{i \in I} x_{i}$ with a new variable $y_{I}$, define $y_{\emptyset} := 1$, and set $x_{i}^{2} = x_{i}$: \begin{equation*}\label{definition:Mk}
    \M^{k}(P):= \left\{ y \in \R^{\binom{n}{\leq k+1}} \, : \, \text{multilinearized inequalities~\eqref{definition:SA_ineq}} \right\}.
    \end{equation*}
\end{enumerate}
We denote by $\SA^{k}(P)$ the (coordinate) projection of $\M^{k}(P)$ onto the space of original variables, i.e.,
\begin{equation*}
    \SA^{k}(P) := \{ x \in \R^{n}  : \text{there exists } y \in \M^{k}(P), \, y_{i} = x_{i}, \, i \in [n] \}.
\end{equation*}
\noindent Let $P_{I}$ be the integer hull of $P$. It is well known that for all $k \in [n]$ we have $P_{I} \subseteq \SA^{k}(P) \subseteq \LS^{k}(P) \subseteq \LS_{0}^{k}(P) \subseteq P$. Moreover, $P_{I} = \LS_{0}^{n}(P)$; see, e.g.,~\cite{cook+dash01,laurent03}.

\smallskip 

Fix $d\leq k \leq n$ and a polytope $P\subseteq [0,1]^n$. The Sherali-Adams hierarchy has the following \emph{locally consistent distribution property}: for every point $y \in \M^k(P)$ and every set of $d$ coordinates $S\subseteq [n]$, we can write $y$ as convex combination of certain points $z^{J_0,J_1}$ (for any $J_0\sqcup J_1=S$) that belong to $\M^{k-d}(P)$ and are integer on coordinates from $S$. The Sherali-Adams hierarchy shares this property with any stronger hierarchy, such as SoS. However, it can be shown that this property \emph{characterizes} the Sherali-Adams hierarchy (see, e.g.,~\cite{fleming+etal19}): $y$ belongs to $\M^k(P)$ if and only if the vectors $z^{J_0,J_1}$ belong to $\M^{k-d}(P)$. To make this property formal, we begin by introducing some notation.

Let $d,k \in \Z_{\geq 0}$ such that $d \leq k$ and $y \in [0,1]^{\binom{n}{\leq k+1}}$. For $S \subseteq [n]$ of cardinality $d$ and a $2$-partition $J_0 \sqcup J_1 = S$, we define:
\begin{equation}\label{eq:z}   y_{I}^{J_{0},J_{1}} := \sum_{H \subseteq J_{0}} (-1)^{|H|} y_{J_{1} \cup I \cup H}, \quad
    z^{J_{0},J_{1}}(y)_{I}:=
    \begin{cases}
    \frac{y^{J_{0},J_{1}}_{I}}{y^{J_{0},J_{1}}_{\emptyset}} & \text{if $y_{\emptyset}^{J_{0},J_{1}} > 0$,} \\
    y^{J_{0},J_{1}}_{I} & \text{if $y_{\emptyset}^{J_{0},J_{1}} = 0$,}
    \end{cases}
\end{equation}
for every $I \subseteq [n]$ with $|I| \leq k-d + 1$.

\begin{proposition}\label{prop:SA_characterization} 
   Let $P\subseteq [0,1]^n$,  $d,k \in \Z_{\geq 0}$ such that $d \leq k$ and $y \in [0,1]^{\binom{n}{\leq k+1}}$. Then $y \in \M^{k}(P)$ if and only if for every $S \subseteq [n]$ of cardinality $d$:
    \begin{enumerate}[label=(\arabic*)]
        \item For every $2$-partition $J_{0} \sqcup J_{1} = S$, we have $z^{J_{0},J_{1}}(y) \in \M^{k-d}(P)$; and
        \item The point $y$ restricted to its first $\binom{n}{\leq k-d+1}$ coordinates can be written as follows 
\begin{equation}\label{prop:SA_characterization:convexcombination}
        y_{I} =
        \sum_{\substack{J_{0} \sqcup J_{1} = S}}
        y_{\emptyset}^{J_{0},J_{1}} z_{I}^{J_{0},J_{1}}(y),
        \end{equation}
        for every $I \subseteq [n]$ with $|I| \leq k-d + 1$.
    \end{enumerate}
\end{proposition}
It is worth emphasizing that in the context of \eqref{prop:SA_characterization:convexcombination}, due to \eqref{eq:z}, it holds that $\sum_{\substack{J_{0} \sqcup J_{1} = S}} y_{\emptyset}^{J_{0},J_{1}} =1$, and thus the vector $y$ is a convex combination of the vectors $\{z^{J_0,J_1}\}_{J_{0} \sqcup J_{1} = S}$. Moreover, vectors $\{z^{J_0,J_1}\}_{J_{0} \sqcup J_{1} = S}$ are integer in the coordinates from $S$.

\subsection{Symmetries} 

Let $\Sym_{n}:=\{\text{bijection }\pi: [n]
\rightarrow [n]\}$ denote the set of all permutations on $[n]$. For $\pi,\pi' \in \Sym
_n$, we denote their (function) composition $\pi'\circ\pi$ as $\pi'\pi$, and by $\pi^{-1}$ the inverse function of $\pi$. We say that a subset $G \subseteq \Sym_{n}$ is a \emph{(permutation) group} if $G$ is nonempty and it is closed under compositions and inverses. We write $H \leq G \leq {\cal S}_n$ to denote that $H,G$ are permutation groups with $H\subseteq G$. The \emph{set of (permutation) symmetries acting on $P$} is defined as
$G(P) :=\{ \pi \in \Sym_{n} \, : \, \pi x \in P \,\text{ for all } x \in P\},$ where $\pi x:=(x_{\pi^{-1}(1)},x_{\pi^{-1}(2)},\ldots,x_{\pi^{-1}(n)}).$
The following fact is well-known and easy to check.
\begin{lemma}\label{lem:group-P}
$G(P)$ is a permutation group.
\end{lemma}
The group $G(P)$ contains (possibly strictly) the set of permutation symmetries of any linear formulation for $P$. We refer to~\cite{margot2010,pfetsch_computational_2019} for details. 

Let $k \in [n]$ and $G\leq {\cal S}_n$. $G$ is \emph{$k$-transitive} if, for every pair of $k$-tuples $t^1$ and $t^{2}$ of $[n]$, where the elements of each tuple are all-different, there exists a permutation $\pi \in G$ such that $\pi(t^{1}_{j}) = t^{2}_{j}$ for $j \in [k]$. We say that $P$ is \emph{$G$-invariant} if $G \leq G(P)$ and it is $k$-transitive if moreover $G$ is $k$-transitive.

Given a subset $S \subseteq [n]$ and $G\leq {\cal S}_n$, we call the subgroup of $G$ which fixes pointwise each element of $S$  the \emph{pointwise stabilizer of $S$} and we denote it by $G_{S} := \{ \pi \in G: \pi(i) = i, \, \forall \pi \in G, \forall i \in S\}$. 
Given any $x \in \mathbb{R}^n$, we define the \emph{orbit of $x$ under $G$} as $\orb_{G}(x) := \{ \pi x \, : \, \pi \in G\}$ where $\pi x := (x_{\pi^{-1}(i)})_{i \in [n]}$. Moreover, we define the \emph{Reynolds operator}~\cite{derksen+kemper02} $\beta_{G}: \R^{n} \to \R^{n}$ as 
    \begin{equation*}
        \beta_{G}(x) := \frac{1}{|G|}\sum_{\pi \in G} \pi x
    \end{equation*}
for $x \in \R^{n}$.  If $x \in P$ for some polytope $P$ and $G\leq G(P)$, then the convexity of $P$ implies that $\beta_{G}(x) \in P$ for every $x \in P$. We prove the following fact in Appendix~\ref{appendix:proof_fixG}.

\begin{proposition}\label{proposition:fixG}
Let $P\subseteq [0,1]^n$ be a polytope, $G\leq G(P)$, and $S$ a subset of $[n]$ of cardinality at most $k$, for some $k \in \mathbb{Z}_{\geq 0}$. Assume that $G$ is $(k+1)$-transitive.
    Then for any $x \in P$, the point $\beta_{G_{S}}(x)$ belongs to  $P$. Moreover, for $i \in [n]$, it has entries:
    \begin{equation*}
        \beta_{G_{S}}(x)_{i} = 
        \begin{cases}
        x_{i} & \text{if $i\in S$}, \\
        \frac{ \sum_{j \in [n]\setminus S} x_{j}} {n - |S|} & \text{if $i \notin S$}.
        \end{cases}
    \end{equation*}
\end{proposition}

In all the definitions given above, we omit the subscript if the group is clear from the context.   For $G\leq {\cal S}_n$ we say that $G$ is \emph{transitive on $T\subseteq [n]$} if for every $i,j \in T$, there exists $\pi \in G$ such that $\pi (i)= \pi(j)$. Thus, in particular, $G$ is $1$-transitive if it is transitive on $[n]$. The proof of the following fact can be found in~\cite[Section 7.1]{dixon+mortimer96}. 
\begin{proposition}\label{proposition:ktransitive}
    A permutation group $G \leq {\cal S}_{n}$ is $(k+1)$-transitive if and only if $G$ is $k$-transitive and for each $S \subseteq [n]$, with $|S| = k$, we have that $G_S$ is transitive on $[n]\setminus S$.
\end{proposition}

We can naturally extend the \emph{permutation group action} of $\mathcal{S}_{n}$ over singletons $\{ i \} \subseteq [n]$ to arbitrary subsets $S \subseteq [n]$ as: $\pi (S) := \{ \pi (i) \, : \, i \in S\}$ for all $\pi \in \mathcal{S}_{n}$. Hence we can define the following \emph{action} on $\R^{\binom{n}{\leq k+1}}$:
\begin{equation}\label{action_on_R^2^n}
(\pi x)_{S} := x_{\pi^{-1}(S)}
\end{equation}
for every $x \in \R^{\binom{n}{\leq k+1}}$, $S \subseteq [n]$ and $\pi \in \mathcal{S}_{n}$. See \cite{dixon+mortimer96} for further background on permutation groups and their actions. The following is an observation that is implicit in the literature, see e.g. Ostrowski~\cite{ostrowski_using_2014}. For completeness, its proof is given in Appendix~\ref{appendix:proof_symmetry-is-inherited}.

\begin{proposition}\label{proposition:symmetry-is-inherited}
Let $G \leq G(P)$ and $k \in \N$. Then $\M^{k}(P)$ is $G$-invariant.
\end{proposition}

A similar statement holds for $\LS_{0}$ and $\LS$. The proof is in Appendix~\ref{app:symmetricOps}.

\begin{proposition}\label{proposition:symmetry_preserving}
    Let $G \leq G(P)$, $k \in \N$. Then $\LS_{0}^k(P)$ and $\LS^k(P)$ are $G$-invariant.
\end{proposition}

\section{Proof of Theorem~\ref{main_theorem}}\label{section:AimpliesB}

The most complex part of the proof is the  \ref{mainthm-a}$\Rightarrow$\ref{mainthm-b} implication, for which we give below a high-level description, and then some details. We then briefly discuss the other implications.

\subsection{\texorpdfstring{Proof of \ref{mainthm-a} $\Rightarrow$ \ref{mainthm-b}}{Proof of \ref{mainthm-a} => \ref{mainthm-b}}}

\paragraph{A high-level description of the proof.} Assume that~\ref{mainthm-a} holds. Let $P$ be a polytope with $G\leq G(P)$ and $G$ $(k+1)$-transitive. We apply Proposition~\ref{prop:SA_characterization} with $d = k$. Thus, we want to exhibit a point $\overline{y}$ in $\M^{k}(P)$, and show that, for every $S\subseteq [n]$ with $|S|=k$, it can be supported by the points ${z}^{J_{0},J_{1}}(\overline{y}) \in \M^{0}(P)$ as in Proposition~\ref{prop:SA_characterization}. However, our construction will proceed in reverse: we build the supporting points first and then show that they satisfy~\eqref{eq:z} and~\eqref{prop:SA_characterization:convexcombination} for some $\overline{y}$. In particular, we perform the following steps.

\smallskip 

\noindent \emph{1.~Construction of points from $\M^0(P)$.}  By leveraging the $(k+1)$-transitivity of $P$, for every $S\subseteq [n]$ with $|S|=k$ and $J_{0} \sqcup J_{1} = S$, we first construct ``highly symmetric'' points, i.e., points $x^{J_0,J_1} \in P$ such that, for any $S'\subseteq[n]$ with $|S'|=k$ and $J_0'\sqcup J_1'$ with $|J_1|=|J_1'|$, we have $x^{J_0,J_1}=\pi x^{J_0',J_1'}$ for some $\pi \in G$. 
\iffalse which satisfy the following:
\begin{itemize}
    \item[a)] 
$x^{J_0,J_1}_i=\begin{cases} 0 & \hbox{for $i \in J_0$,} \\  1 & \hbox{for $i \in J_1$,}\\ \Delta_{|J_1|} & \hbox{otherwise.}\end{cases}$
\item[b)] For every $J_0\sqcup J_1= J_0'\sqcup J_1'$ with $|J_1|=|J_1'|$, we have $x^{J_0,J_1}=\pi x^{J_0',J_1'}$ for some $\pi \in G$.
\end{itemize}
\fi 
By definition, $(1,x^{J_0,J_1}) \in \M^0(P)$. 

\smallskip 

\noindent \emph{2.~System of constraints defining $\overline y$.} Note that any point $\overline y$ that satisfies~\eqref{eq:z} and~\eqref{prop:SA_characterization:convexcombination} with respect to our constructed vectors $z^{J_0,J_1}:=(1,x^{J_0,J_1})$ is guaranteed to belong to $\M^k(P)$. Imposing these conditions leads to a system of equations whose variables are the coordinates of $\overline{y}$. This system appears to be highly complex and non-tractable. However, because of the symmetric structure of $z^{J_0,J_1}$, we look for solutions $\overline{y}$ that are highly symmetric. Indeed, for every $I,J\subseteq [n]$ with $|I|=|J|\leq k+1$ and every $J_0\sqcup J_1=S$, $J_0'\sqcup J_1'=S'$ with $|J_{1}| = |J_{1}'|$ and $|S|=|S'|=k$, we impose the following conditions:
\vspace{-0.2em}
\begin{enumerate}[label=(C-\arabic*), leftmargin=*]
    \item $\overline y_I = \overline y_J =: \overline \gamma_{|I|}$;
    \item $\overline y_{\emptyset}^{J_{0},J_{1}} = \overline y_{\emptyset}^{J'_{0},J'_{1}}=:\overline \lambda_{|J_1|}$. 
\end{enumerate}
Conditions (C-1) and (C-2) above lead to a (much smaller) system in the variables $\gamma$ and $\lambda$, which have $k+1$ components each. Thus, we are left to show that this system has a solution. This is the most technical part of the proof, divided into two steps.

\smallskip

\noindent{\emph{3.~Solution to the system defining $\overline \gamma$.} We first consider a subsystem of linear equalities in the variables $\gamma$ only, and show that the corresponding constraint matrix is invertible. 

\smallskip 

\noindent{\emph{4.~Solution to the general system.} Based on the previous step, we give an explicit formula for $\overline \gamma$ and $\overline \lambda$, and proceed to show that they satisfy all constraints of the system in $\gamma$ and $\lambda$ defined in Step~2.

\medskip 
We now present full details on Step 1.~above, and the main lemma leading to Step 2. The proof of this lemma, as well as the remaining steps of the proof, appear in Appendix~\ref{app:proof-main-theorem}. 

\paragraph{1.~Construction of points from $\M^0(P)$.}

For $\ell \in \{0,\dots,k\}$, define 
\begin{equation*}
    P_{\ell} := P \cap \{x \in [0,1]^{n} \, : \, x_{i} = 1 \, \text{for all } i \in [\ell], x_{j} = 0 \, \text{for all } j \in [k]\setminus [\ell] \}, 
    \end{equation*}
and observe that it is non-empty because of Assumption~\ref{mainthm-a}. Thus, let $x^\ell \in P_{\ell}$. Recall that $G_{[k]}$ is the pointwise stabilizer of $[k]$. Observe that $P_{\ell}$ is $G_{[k]}$-invariant. Moreover, since $G$ is $(k+1)$-transitive, by Proposition~\ref{proposition:ktransitive}, $G_{[k]}$ is transitive on $[n]\setminus [k]$. Therefore, by applying Proposition~\ref{proposition:fixG} to $x^{\ell}$ with $G_{[k]}$ and $P_{\ell}$ we obtain
\begin{equation}\label{eq:xell}
\overline{x}^\ell := \beta_{G_{[k]}}(x^{\ell}) = (\underbrace{1,\dots,1}_\ell,\underbrace{0,\dots,0}_{k-\ell},\underbrace{\Delta_\ell,\dots,\Delta_\ell}_{n-k}) \in P_{\ell},
\end{equation} 
for some $\Delta_\ell \in [0,1]$. 
Define $\Delta := (\Delta_{0},\dots,\Delta_{k}) \in [0,1]^{k+1}$ and observe that $\Delta \in (0,1)^{k+1}$ since $P$ is integer-empty. 

Next, let $S\subseteq [n]$ with $|S|=k$ and consider an arbitrary $2$-partition $J_{0} \sqcup J_{1} = S$ such that $|J_{1}| = \ell$. By definition, the point $\overline{x}^{\ell}$ has $\ell$ ones in entries indexed by $[\ell]$ and zeros in entries indexed by $[k]\setminus [\ell]$. Since $G$ is $(k+1)$-transitive and $G\leq G(P)$, there exists some permutation $\pi \in G$ and some point $\overline x^{J_{0},J_{1}} \in P$ such that $\pi$ maps: all of the ones in positions $[\ell]$ of $\overline{x}^{\ell}$ to entries indexed by $J_{1}$ of $\overline x^{J_{0},J_{1}}$; all of the zeros in positions $[k]\setminus[\ell]$ to entries indexed by $J_{0}$; and the fractional $\Delta_{\ell}$'s in positions $[n]\setminus[k]$ to entries indexed by $[n]\setminus S$. In other words,
\begin{equation}
    \overline x^{J_{0},J_{1}}_{i} := (\pi \overline{x}^{\ell})_{i} = 
    \begin{cases}
    0 & \text{if $i \in J_{0}$}, \\
    1 & \text{if $i \in J_{1}$}, \\
    \Delta_{\ell} & \text{if $i \in [n] \setminus S$}.
    \end{cases}
\end{equation}
By definition of $\M^0(P)$, we immediately deduce the following. \begin{lemma}\label{cl:z-in-M0}
For each $S\subseteq [n]$ with $|S|=k$ and $2$-partition $J_{0} \sqcup J_{1} = S$, we have $(1, \overline x^{J_0,J_1}) \in \M^0(P)$.
\end{lemma}

\paragraph{2.~System of constraints defining $\overline y$.} 
We give next a sufficient conditions for the existence of $\overline y \in \M^d(P)$. For $t \in \mathbb{N}$ and $\Theta=(\Theta_0,\dots,\Theta_t) \in \R^{t+1}$, we let $A^t_\Theta \in \R^{(t+1)\times(t+1)}$ with rows indices in $\{0,\dots,t\}$ and column indices in $\{1,\dots, t+1\}$ be defined as follows:
\begin{equation}\label{linearsystem:A_definition-correct}
    (A^{t}_{\Theta})_{\ell,r} := \left\{ \begin{array}{crc}
       (-1)^{\ell + r-1} \left[ \binom{t-\ell}{r-\ell-1} + \binom{t-\ell}{r-\ell} \Theta_{\ell} \right] & \mbox{ if $\ell < r$}, \\
       -\Theta_{\ell}  & \mbox{if $\ell = r$},  \\
       0  & \mbox{if $\ell >r$.}
    \end{array}
    \right.
\end{equation}
}

See Example~\ref{ex:matrix-A}} in Appendix~\ref{app:proof-main-theorem} for an example of Matrix $A$.

\begin{lemma}\label{cl:system-in-gamma}
Assume that the following system in the variables $\gamma_0,\dots, \gamma_{k+1} \in \mathbb{R}$, $\lambda_0,\dots, \lambda_{k} \in \mathbb{R}$ has a solution 
\begin{subequations}\label{system-in-gamma}
\begin{align}
        A^{k}_{\Delta} \gamma &= \Delta_{0} (1, 0, \dots, 0), \label{eq:z_in_M0-new} \\
        1 &=
        \sum_{\ell=0}^{k} \binom{k}{\ell} \lambda_{\ell}, \label{eq:convex_combination_emptyset-new} 
        \\
        {\gamma}_{1} &= \sum_{\ell=1}^{k} \binom{k-1}{\ell-1} \lambda_{\ell}, \label{eq:convex_combination_singletons:integer-new} \\
        {\gamma}_{1} &= \sum_{\ell=0}^{k} \binom{k}{\ell} \lambda_{\ell} \Delta_{\ell},  \label{eq:convex_combination_singletons:fractional-new} 
        \\
        \lambda_{\ell} &= \sum_{r =0}^{k-\ell} (-1)^{r} \binom{k-\ell}{r} \gamma_{r + \ell} \quad \hbox{$\forall \ell \in \{0,1,\dots,k\}$}, \label{eq:ugly-inequality} \\
        & {\zeros} \leq \lambda, \gamma  \leq {\ones}, \quad \gamma_{0} = 1 \label{eq:lambda-gamma-nonnegative},
\end{align}
\end{subequations}
that is strictly positive in all components, where we abbreviate $\gamma=(\gamma_1,\dots,\gamma_{k+1})$.

Then, $\overline y \in \M^k(P)$, where \begin{equation}\label{eq:overliney}\overline y_I = \overline \gamma_{|I|} \quad \hbox{for all $I\subseteq [n]$, $|I|\leq k+1$}.\end{equation}
\end{lemma}

As discussed above, the proof of Lemma~\ref{cl:system-in-gamma} and the existence of a solution to the system are given in Appendix~\ref{app:proof-main-theorem}. 

\subsection{Proofs of the remaining statements from Theorem~\ref{main_theorem}} We note that (B) implies (C) and (C) implies (D) are well-known facts, see, e.g.,~\cite{lovasz+schrijver91,laurent03}. We show that (D) implies (A). 
Suppose therefore that $\LS_{0}^{k}(P)$ is nonempty and let $G$ be a $(k+1)$-transitive subgroup of $G(P)$. We make use of the following lemma, proved in Appendix~\ref{app:(D)=>(A)}.

\begin{lemma}\label{lemma:D_implies_A_2}
Let $j \in [k]$ and $i \in [n]$. Suppose that $P$ is $(k+1)$-transitive and that there exists $x \in \LS^{k-j+1}_{0}(P)$ such that $x_{i} \in (0,1)$. Then for every subset $J \subseteq [n]$ with $i \in J$ and  $|J| = j$, there exists $x^{i,1},x^{i,0} \in \LS^{k-j}_{0}(P)$ such that: a) $x_{i}^{i,1} = 1$ and $x_{i}^{i,0} = 0$; b) $x^{i,1}_{\ell} = x^{i,1}_{r}$ and $x^{i,0}_{\ell} = x^{i,0}_{r}$ for all $\ell,r \in [n] \setminus J$. 
\end{lemma}

 Let $x \in \LS_{0}^{k}(P)$ and consider $\overline x:=\beta_G(x)$. Since $P$ is integer-empty, we deduce that $\overline x$ is fractional in all components. By Proposition~\ref{proposition:symmetry_preserving}, $\LS_0^k(P)$ is $G$-invariant, hence $\overline x \in \LS_{0}^k(P)$.  Pick $i=1$ and apply Lemma~\ref{lemma:D_implies_A_2} as to obtain points $x^{1,1}$ and $x^{1,0} \in \LS_0^{k-1}(P)$. Let $G_1$ be the subgroup of $G$ that fixes the first coordinate. Using again the fact that $P$ is integer-empty, we deduce that all coordinates of $\overline x^{1,1}:=\beta_{G_1}(x^{1,,1})$ are fractional, except the first that is equal to $1$. Similarly, all coordinates of $\overline x^{1,0}:=\beta_{G_1}(x^{1,0})$ are fractional, except the first that is equal to $0$. Since $\LS^{k-1}(P)$ is $G$-invariant by Proposition~\ref{proposition:symmetry_preserving}, we have $\overline x^{1,0}, \overline x^{1,1} \in \LS^{k-1}(P)$. Iterating, for all $j \in [k]$ and vectors $q^j \in \{0,1\}^j$ we obtain points $\overline x^{q^2}, \overline x^{q^3}\dots, \overline x^{q^k}$ such that 
 \begin{itemize}[label=$\bullet$]
     \item $\overline x^{q^j} \in \LS_0^{k-j}(P)$;
     \item $\overline x^{q^j}_\ell = q^j_\ell$ for all $\ell \in [j]$.     
 \end{itemize}
 In particular, for all binary $k$-tuples $q$, we obtain a point $\overline x^{q} \in \LS_0^{0}(P)=P$ with $\overline x^{q}_\ell =q_\ell$ for all $\ell \in [k]$.
 Now fix a $(n-k)$-dimensional face $F$ of $[0,1]^n$, defined as:
 $$
 F=\{ x \in [0,1]^n : \, x_i = 1 \, \forall i \in S_1; \, x_i = 0 \, \forall i \in S_0\}.  
 $$
 Let $q \in \{0,1\}^k$ such that $q_1=q_2=\dots=q_{|S_1|}=1$ and $q_{|S_1|+1}=q_{|S_1|+2}=\dots=q_{k}=0$. Since $G$ is $(k+1)$-transitive, there exists $\pi \in G$ that maps $\overline x^q$ to a point of $\pi \overline x^q \in F$. In particular, $\pi \overline x^q \in F \cap P$, as required.

\section{Conclusion}\label{section:conclusion}

In this paper, we investigate bounds on the number of iterations that LP-based hierarchies need to certify integer-emptiness of a polytope $P$ contained in the binary cube.  Roughly speaking, the ``more symmetric'' $P$ is, the larger is the number of rounds for which the hierarchies ``collapse'' -- that is, they can either all certify integer-emptiness, or none can. An important feature of our work is that we do not assume any particular structure of the polytopes, beyond the assumption of symmetry. We illustrate how certain results from the literature and new insights can be deduced from our main result. It is a natural question to investigate the behavior under various (partial) symmetries of LP hierarchies and other automatic approaches, such as the SoS (Lasserre) hierarchy, as well as the Split and Chv\`atal-Gomory closures. Finally, while our results suggest that some of the current methods for tightening LP relaxations behave poorly for symmetric instances, it is a relevant task to develop new techniques that are enhanced--or at least not negatively affected--by symmetry. 

\bigskip

\noindent {\bf Acknowledgments.} Yuri Faenza and Matias Villagra acknowledge support from the AFOSR grant FA9550-23-1-0697. José Verschae thanks the support of ANID FONDECYT Regular Nr. 1221460 and the Center for Mathematical Modeling (CMM) Basal fund FB210005.
Victor Verdugo thanks the support of ANID FONDECYT Regular Nr. 1241846.

% ---- Bibliography ----
\bibliographystyle{splncs04}
\bibliography{citations}

\appendix

\section{On the hypotheses of Theorem~\ref{main_theorem}}\label{appendix:hypothesis-thm}

We illustrate here how Theorem~\ref{main_theorem} is tight in several ways. 
First, we claim that we cannot weaken the $(k+1)$-transitivity assumption to $k$-transitivity. Consider the polytope $P := \conv\left(\{ \pi (1,1/2,0), \pi (1,1,1/10) \, : \, \pi \in C_{3}\} \right)$, where $C_{3}$ is the cyclic group acting on the set $[3]$. This polytope is $1$-transitive but not $2$-transitive. A routine computation shows that $\SA(P) = \emptyset$ while using the definition of $\LS_0$, one can check that $\LS_0(P)\neq \emptyset$, see Fig.~\ref{fig:orb}. Thus,~\ref{mainthm-b} does not hold, while~\ref{mainthm-d} holds.

\begin{figure}[H]
\begin{center}
\includegraphics[width=0.3\textwidth]{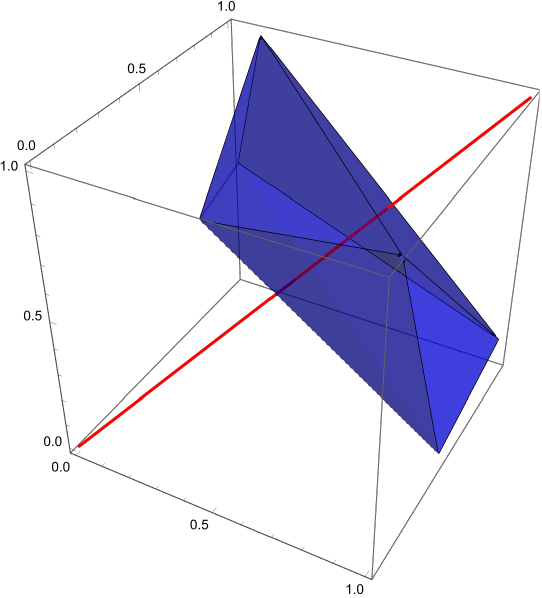}
\hspace{2cm}
\includegraphics[width=0.3\textwidth]{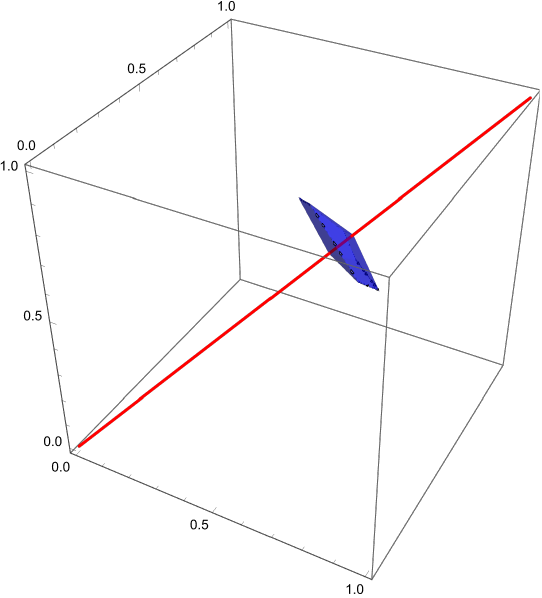}
\caption{On the left, the polytope $P=\conv\left(\{ \pi (1,1/2,0), \pi (1,1,1/10) \, : \, \pi \in C_{3}\} \right)$ from Appendix~\ref{appendix:hypothesis-thm}. On the right, $\LS_{0}(P)$.}\label{fig:orb}
\end{center}
\end{figure}

Next, observe that we cannot drop the integer-emptyness assumption, since trivially~\ref{mainthm-b},~\ref{mainthm-c},~\ref{mainthm-d} hold for any polytope with an integer point, irrespectively of whether~\ref{mainthm-a} holds. 

Last, recall that we already argued that $\SA^1(P)$ and $\LS_0^{1}(P)$ may differ (while both being non-empty), even if the hypothesis of the theorems hold for $k=1$; see Example~\ref{ex:3D}.

\section{Examples and applications}

\subsection{Steiner Triple Covering Problems}\label{app:example:steiner}

We define and discuss here the instances STS$_n$ of the Steiner Triple Cover Problem. STS$_1$ is defined as the following set covering instance
$$\min \{\ones^T x \, : \, x_1 + x_2 + x_3 \geq 1, \, x \in \{0,1\}^3\}. $$
For $n \in \N_{\geq 2}$, we define STS$_n$ as a function of STS$_{n-1}$ as follows. We let $A_{n-1}$ be the constraint matrix for STS$_{n-1}$. We let $I_{n-1}$ be the identity matrix with $3^{n-1}$ rows and columns. We let $B_{n}$ be the matrix with $3^{n}$ columns and whose rows are all and only the vectors that satisfy the following: (a) three entries $i_0,i_1,i_2$ are equal to $1$, and the remaining entries are equal to $0$; (b) $3^j+1\leq i_j\leq 3^{j+1}$ for $j =0,1,2$; (c) $i_j \mod 3^n$ assume three different values for $j \in \{0,1,2\}$. STS$_n$ is then defined as the following set covering problem (for the sake of readability, we omit the $0$s):
\[
\begin{array}{lrl}
\displaystyle \min \;\ones^T x \\ 
\text{s.t.} &
\begin{pmatrix}
A_{n-1}  &  \\
 & A_{n-1} &  \\
 &  & A_{n-1} \\
I_{n-1} &  &  \\
 & I_{n-1} &  \\
 &  & I_{n-1} \\
\multicolumn{3}{c}{B_{n}} \\
\end{pmatrix} x &  \geq \ones \\ [1em]
& x & \in \{0,1\}^{3^n}.
\end{array}
\]

\begin{example} STS$_2$ is given by the following binary integer program.
\[
\begin{array}{rrrrrrrrrrl}
\min & x_1 & + x_2 & + x_3 & + x_4 & + x_5 & + x_6 & + x_7 & + x_8 & + x_9 \\ [1em]
& {x_1} & + {x_2} & + {x_3} & & & & & & & \geq 1 \\ [.5em]
&  & & & {x_4} & + {x_5} & + {x_6} & &&  & \geq 1 \\ [.5em]
&  & & &&&&  {x_7} & + {x_8} & + {x_9} & \geq 1 \\ [.5em]
& {x_1} & & &  + {x_4} & & & + {x_7} & & & \geq 1  \\ [.5em]
& & {x_2} & & & + {x_5} & & & + {x_8} & & \geq 1 \\ [.5em]
& & & {x_3} & & & + {x_6} & & & + {x_9} & \geq 1 \\ [.5em]

& {x_1} & & & & + {x_5} & & & & + {x_9}  & \geq 1 \\ [.5em]
& {x_1} & & & & & + {x_6} & & + {x_8} & & \geq 1 \\ [.5em]
& & {x_2} & & + {x_4} & & & & & + {x_9} & \geq 1 \\ [.5em]
& & {x_2} & & & & + {x_6} & + {x_7} & & & \geq 1 \\ [.5em]
& & & {x_3} & + {x_4} & & & & + {x_8}  & & \geq 1 \\ [.5em]
& & & {x_3} & & + {x_5} & & + {x_7} & & & \geq 1 \\ [.5em]
& & & & & & & & & x & \in \{0,1\}^9.
\end{array}
\]
An optimal solution to this instance is, e.g., $x^*=(1,1,0,1,0,1,0,1,0)$ with objective value $5$.
\end{example}
Now fix $n \in \N_{\geq 2}$. By construction, $x_1+x_2+x_3 \geq 1$ is a valid constraint of STS$_n$. It is not hard to see that there exists a feasible solution $\overline x$  for STS$_n$ that sets at least $3$ variables to $0$. If the feasible region of the Linear Programming relaxation of STS$_n$ were $3-$transitive, then there would exist a feasible solution $\hat x$ such that $\hat x_1=\hat x_2 = \hat x_3=0$, violating $x_1+x_2+x_3\geq 1$.

We now show how to obtain a $2$-transitive integer-empty optimal LP face for STS$_2$ that satisfies (A) from Theorem~\ref{main_theorem}. A similar argument applies for larger values of $n$. Add the cut $\sum_{i=1}^n x_i \geq 4$, whose intersection with the Linear Programming relaxation of STS$_2$ gives the integer-empty, optimal LP face $F$. Following~\cite{hall1960automorphisms}, it is easy to see that $F$ is $2$-transitive. Thus, we are left to show that $F$ intersects all $(n-1)$-dimensional faces of the hypercube. By symmetry, we only need to consider $F_i= \{ x \in [0,1]^9 : x_1 = i\}$ for $i=0,1$. We then conclude the argument by observing that
$$
\left(0,1,0,\frac{1}{2},\frac{1}{2},\frac{1}{2},\frac{1}{2},\frac{1}{2},\frac{1}{2}\right) \in F \cap F_0
$$
and 
$$
\left(1,0,0,\frac{1}{2},\frac{1}{2},\frac{1}{2},\frac{1}{2},\frac{1}{2},\frac{1}{2}\right) \in F \cap F_1.
$$
\subsection{Knapsack-Cover Inequalities}\label{app:KCover}
In this section, we show an application of our main techniques. The knapsack cover problem consists of $n$ items, each with a covering capacity $u_i\ge 0$ and a cost $c_i\ge 0$. Our aim is to select a subset of items of minimum cost whose total covering capacity is at least a given demand $D\ge 0$, i.e., 
\[\min \left\lbrace \sum_{i=1}^n c_ix_i: \sum_{i=1}^n x_iu_i\ge D, x\in \{0,1\}^n\right\rbrace.\] The straightforward linear relaxation, where we replace the constraint $x\in \{0,1\}^n$ by $x\in [0,1]^n$, has an unbounded integrality gap. The knapsack-cover inequalities, introduced by Wolsey~\cite{wolsey_faces_1975} and later popularized by Carr et al.~\cite{carr_strengthening_2000}, correspond to a set of valid inequalities that reduced the integrality gap to 2, and are given by

\[ \sum_{i\not\in S} x_i\bar{u}_i(S)\ge D(S) \qquad \text{for all } S\subseteq [n],
\]
where $D(S) := \max \{D-\sum_{i\in S}u_i,0\}$ and $\bar{u}_i(S):= \min\{D(S),u_i\}$. 

An important open problem, particularly relevant for approximation algorithms, is how to obtain relaxations with smaller integrality gaps. Bienstock and Zuckerberg~\cite{daniel_bienstock_approximate_2003,bienstock_subset_2004} showed that even after a sublinear number of rounds $k=o(n)$ of the Sherali-Adams hierarchy (as well as for the semidefinite strengthening $\LS_+$ of the Lovász-Schrijver hierarchy), the integrality gap still remains 2. Later, Kurpisz et al.~\cite{kurpisz_sum--squares_2020} proved that even after $\Omega(\log(n)^{1-\varepsilon})$ (for any $\varepsilon>0$) rounds of the stronger Sum-of-Squares (Lasserre) SDP hierarchy, the integrality gap remains 2. 

In what follows we show how to apply our main construction from Section~\ref{section:AimpliesB} yields the same result as Bienstock and Zuckerberg~\cite{daniel_bienstock_approximate_2003,bienstock_subset_2004} for the Sherali-Adams relaxation. Unlike their line of thought, that needs to propose the full vector $y\in \M^k(P)$ from scratch, we just need to give vectors that belong to $P$ intersected to each of the $(n-k)$-dimensional faces of the hypercube; a significantly simpler task. 

Consider the instance of knapsack cover where $u_i=1$ and $c_i=1$ for all $i\in [n]$, and $D=1+\frac{1}{n-1}$. The linear relaxation with knapsack-cover inequalities\footnote{As Kurpisz et al.~\cite{kurpisz_sum--squares_2020} observed, the inequality $\sum_{i=1}^n x_i\ge 1+\frac{1}{n-1}$ is redundant in this formulation and thus it can be omitted, which yields the polytope studied by Bienstock and Zuckerberg~\cite{daniel_bienstock_approximate_2003,bienstock_subset_2004}.} for this instance is
\[\min\left\lbrace\sum_{i=1}^n x_i: \sum_{i=1}^n x_i\ge 1+\frac{1}{n-1}, \sum_{j\in [n]\setminus\{i\}}x_j\ge 1 \text{ and }0\le x_i\le 1\text{ for all }i\in [n]\right\rbrace.\]

Let $P$ denote the feasible region of this linear program. We will show $\SA^k(P)$ contains a solution $x$ with $\sum_{i=1}^nx_i = 1+O(\frac{k}{n})$. As the optimal integral solution of this LP has value 2,
this shows that if $k=k(n)$ grows sublinearly on $n$ (for example, $k=n^{1-\delta}$ for some $\delta>0$) then $\sum_{i=1}^nx_i$ converges to 1, hence implying that the integrality gap remains 2. 

\begin{proposition}
For any $k\le n/2$,  there exists a vector $x\in \SA^k(P)$ with value $\sum_{i=1}^nx_i = 1+O(\frac{k}{n})$. 
\end{proposition}
\begin{proof}
Consider any number $k\le n/2$. We obtain an explicit vector in $\SA^k(P)$ by using the construction from Section~\ref{section:AimpliesB}. We follow the notation from that section. Observe first that the polytope $P$ is $\mathcal{S}_n$-invariant. Hence $P$ is $n$-transitive, and therefore $(k+1)$-transitive. Observe that the following vectors belong to $P$:
\begin{align*}
\overline{x}^{0}&:= \left(\underbrace{0,\ldots,0}_{k},\underbrace{\frac{1}{n-k-1},\ldots,\frac{1}{n-k-1}}_{n-k}\right),\\
\overline{x}^{1}&:= \left(\underbrace{1,0,\ldots,0}_{k},\underbrace{\frac{1}{n-k},\ldots,\frac{1}{n-k}}_{n-k}\right),\\
\overline{x}^{2}&:= (\underbrace{1,1,0,\ldots,0}_{k},\underbrace{0,\ldots,0}_{n-k}),\\
\vdots\\
\overline{x}^{k}&:= (\underbrace{1,1,1,\ldots,1}_{k},\underbrace{0,\ldots,0}_{n-k}).
\end{align*}
Hence $\Delta_0=\frac{1}{n-k-1}$, $\Delta_1=\frac{1}{n-k}$ and $\Delta_{\ell}=0$ for all $\ell=2,\ldots,k$.  We use \eqref{linearsystem:lambda_solution} in order to compute the values $(\overline{\lambda}_\ell)_{\ell=0}^k$. Let us denote $D:=\omega^k_\Delta$ to avoid excessive notation. Then
\begin{align*}
\overline{\lambda}_0 & = \frac{1}D \prod_{i=1}^{k}(1-\Delta_i) = \frac{1}{D}\left(1-\frac{1}{n-k}\right),\\
\overline{\lambda}_1 & = \frac{1}D \Delta_0 \prod_{i=2}^{k}(1-\Delta_i) = \frac{1}{D} \left( \frac{1}{n-k-1} \right),\\
\overline{\lambda}_2 & = \frac{1}D \Delta_0\Delta_1 \prod_{i=3}^{k}(1-\Delta_i) = \frac{1}{D} \left( \frac{1}{n-k-1} \right) \left( \frac{1}{n-k} \right),
\end{align*}
and $\overline{\lambda_\ell}=0$ for all $\ell\ge 3.$ We can easily compute $D$ by using \eqref{eq:convex_combination_emptyset-new} and solving for $D$, which yields

\begin{align*}D &=  D\overline{\lambda}_0+Dk\overline{\lambda}_1+D\frac{k(k-1)}{2} \overline{\lambda}_2 \\
&= \left(1-\frac{1}{n-k}\right) + k \left( \frac{1}{n-k-1} \right)  + \left( \frac{k(k-1)}{2} \right) \left( \frac{1}{n-k-1} \right) \left( \frac{1}{n-k} \right)\\
&= 1  -\frac{1}{n-k} + \frac{k}{n-k-1}  + \frac{1}{2} \left(\frac{k}{n-k-1} \right) \left( \frac{k-1}{n-k} \right) \\
&= \frac{2 (n-k-1)(n-k) - 2 (n-k-1) + 2k (n-k) + k(k-1)}{2 (n-k-1)(n-k)} \\
&= \frac{2n^{2} - 2nk - 4n + k^{2} + 3k + 2}{2 (n-k-1)(n-k)}.
\end{align*}
We obtain that the vector $x := \overline{\gamma}_{1} \ones$ belongs to $\SA^{k}(P)$, where $\overline{\gamma}_{1}$ can be computed using~~\eqref{eq:convex_combination_singletons:integer-new}
\begin{align*}
    \overline{\gamma}_{1} &= \sum_{\ell=1}^{2} \binom{k-1}{\ell-1} \overline{\lambda}_{\ell} \nonumber\\
    &= \frac{1}{D} \left(\frac{1}{n-k-1} \right) + \frac{1}{D} \left( \frac{1}{n-k-1} \right) \left( \frac{1}{n-k} \right)(k-1) \nonumber\\
    &= \frac{n-1}{D (n-k-1)(n-k)}.
\end{align*}
Therefore, the objective value of $x$ is
\begin{align*}
    \sum_{i=1}^{n} x_{i} &= \frac{n(n-1)}{D (n-k-1)(n-k)} \\
    &= \frac{2 n (n-1)}{2n^{2} - 2nk - 4n + k^{2} + 3k + 2} \\
    &= \frac{1 - (1/n)}{1 + (1/2)(k/n)^{2} - (k/n) + (3/2)(k/n^{2}) - (2/n) + (1/n^{2})} \\
    &= \frac{1}{1-\Theta(k/n)} = 1 + \Theta\left( k / n\right).
\end{align*}
\end{proof} 

\section{Additional facts and proofs from Section~\ref{section:preliminaries}}

\subsection{A lemma}

The following lemma is useful in our analysis.

\begin{lemma}\label{lem:cancels-out}
Let $d \in \mathbb{Z}_{\geq 0}$ and $y \in [0,1]^{\binom{n}{\leq d}}$. For every $I \subseteq [n]$ with $|I| \leq 1$, $S \subseteq [n]$ of cardinality $d$, every $2$-partition $J_{0} \sqcup J_{1} = S$, and every $j \in J_0$, we have 
$$\sum_{H \subseteq J_0}(-1)^{|H|}y_{J_1\cup \{j\} \cup H}=0.$$
\end{lemma}
\begin{proof}
We have 
\begin{align*}
 &\sum_{H \subseteq J_0}(-1)^{|H|}y_{J_1\cup \{j\} \cup H }\\
 & =  \sum_{H \subseteq J_0 : j \in H}(-1)^{|H|}y_{J_1\cup \{j\} \cup H }+ \sum_{H \subseteq J_0 : j \notin H}(-1)^{|H|}y_{J_1\cup \{j\} \cup H } \\
 & = \sum_{H \subseteq J_0 : j \notin H}(-1)^{|H|+1}y_{J_1\cup \{j\} \cup H }+ \sum_{H \subseteq J_0 : j \notin H}(-1)^{|H|}y_{J_1\cup \{j\} \cup H }=0.\qedhere
\end{align*}    
\end{proof}

\subsection{Proof of Proposition~\ref{proposition:fixG}}\label{appendix:proof_fixG}
The statement is obvious for $i \in S$. We next claim $\beta_{G_S}(x)_j=\beta_{G_S}(x)_i$ for any $i,j \in [n]\setminus S$. Thus, let $i,j \in [n]\setminus S$. Since $G$ is $(k+1)$-transitive, there exists $h \in G$ such that $h(\ell)=\ell$ for $\ell \in S$ and $h(j)=i$. By construction, $\pi \in G_S$. Thus, using the well-known fact that $G_S$ is a subgroup of $G$, we can  write 
\begin{align*}
        \beta_{G_S}(x)_{j} & = \frac{1}{|G_S|} \sum_{\pi \in G_S} (\pi x)_j\\
        &= \frac{1}{|G_S|} \sum_{\pi \in G_S} x_{\pi^{-1}(j)} \\ 
        &= \frac{1}{|G_S|} \sum_{\pi \in G_S} x_{\pi^{-1}h^{-1}(i)}\\
        &= \frac{1}{|G_S|} \sum_{\pi \in G_S} x_{(h\pi)^{-1}(i)} = \frac{1}{|G_S|} \sum_{r \in G_S} x_{r^{-1}(i)}  = \frac{1}{|G_S|} \sum_{r \in G_S} x_{r(i)}= \beta_{G_S}(x)_{i}.   \end{align*}
     Therefore, for $ i \in [n]\setminus S$, we have $$\sum_{j \in [n]\setminus S}x_j + \sum_{j \in S} x_j=\|x\|_1=\|\beta_G(x)\|_1=\sum_{j \in [n]} \beta_G(x)_j= (n-|S|) \cdot \beta_G(x)_i + \sum_{j \in S} x_j,$$
from which the statement follows. \hfill $\qed$

\subsection{Properties of Symmetries of Polytopes}
The results of this subsection will be used for the proof of Proposition~\ref{proposition:symmetry-is-inherited}. 
We will use the following characterization of the symmetries of a polytope $P$.
\begin{lemma}\label{lm:CharacterizationG(P)}
Let $P$ be a polytope. A permutation $\pi\in \mathcal{S}_n$ belongs to $G(P)$ if and only if for any valid inequality $\alpha^\top x \ge \beta$ of $P$ we have that $(\pi \alpha)^{\top} x \ge \beta$ is also valid for $P$.
\end{lemma}
\begin{proof}
Consider a permutation $\pi\in G(P)$ and a valid inequality $\alpha^\top x \ge \beta$ for $P$. We must show that  $(\pi \alpha)^{\top} x \ge \beta$ is also valid. Since $G(P)$ is a group, then $\pi^{-1}\in G(P)$. If $x\in P$ then $\pi^{-1} x\in P$, and thus $\alpha^\top \pi^{-1}x \ge \beta$. Observing that $\alpha^\top (\pi^{-1}x) = (\pi \alpha)^{\top} x$, then $(\pi \alpha)^{\top} x \ge \beta$ for all $x\in P$. Hence $(\pi \alpha)^{\top} x \ge \beta$ is a valid inequality for $P$. 

To show the other implication, consider a permutation $\pi$ such that for any valid inequality $\alpha^\top x \ge \beta$ for $P$, we obtain that $(\pi \alpha)^{\top} x \ge \beta$ is also valid. It suffices to show that for any $x\in P$ we have that $\pi x\in P$.
Consider any formulation of $P$ given by $P=\{x: Ax\ge b\}$, for some matrix $A\in \mathbb{R}^{m\times n}$ and $b\in\mathbb{R}^m$. Let $x\in P$. Then $a_i^\top x\ge b_i$ for all $i\in [m]$, where $a_i$ is the transpose of the $i$-th row of $A$. Then $(\pi a_i)^{\top} x \ge b_i$. Therefore $(\pi a_i)^{\top} x = a_i^{\top} (\pi^{-1}(x))\ge b_i$ for all $i\in [m]$, and thus $\pi^{-1}(x)\in P$. We conclude that $\pi^{-1}\in G(P)$. As $G(P)$ is a group, it is closed under taking inverses, and thus $\pi\in G(P)$.
\end{proof}

\begin{lemma}\label{lemma:SAaddIneq}
Let $Ax \geq b$ be a system of linear inequalities containing $0 \leq x \leq \ones$, and let $\alpha^{\top} x \geq \beta$ be an inequality implied by $Ax \geq b$. Then
\begin{equation*}
    \M^k(\{Ax \geq b\}) = \M^k(\{Ax \geq b, \alpha^{\top} x \geq \beta\}).
\end{equation*}
\end{lemma}
\begin{proof}
    Clearly $\M^k(\{Ax \geq b, \alpha^{\top} x \geq \beta\}) \subseteq \M^k(\{Ax \geq b\})$. We focus then in the other inclusion. 
    
    Consider the linearization of the  inequality $\prod_{i \in J_{1}} x_{i} \, \prod_{j \in J_{0}} (1-x_{j})(\alpha^{\top} x - \beta) \geq 0$, i.e.,
   \begin{align}
        \left( \sum_{j \in J_{1}} \alpha_{j} - \beta \right)
          f_{k}(J_{0}, J_{1})
        + \sum_{j \in [n] \setminus S}
          \alpha_j\, f_{k+1}(J_{1} \cup \{j\}, J_{0})
        &\ge 0, \label{eq:Mk_alphabeta}
    \end{align}
    where the value $f_{k}(J_{0},J_{1})(y) = \sum_{H \subseteq J_{0}} (-1)^{|H|} y_{J_{1} \cup H}$ is the multilinearization of $ \prod_{i \in J_{1}} x_{i} \prod_{j \in J_{0}} (1 - x_{j})$ when replacing $y_{I} := \prod_{i \in I} x_{i}$ for all $I \subseteq [n]$. Now we show that \eqref{eq:Mk_alphabeta} is implied by the inequalities of $\M^k(\{Ax \geq b\})$. Indeed, observe that there exists a vector of duals $u \geq 0$ such that $u^{\top} A = \alpha$ and $u^{\top} b = \beta$ (provided the polytope induced by $Ax \geq b$ has positive dimension, see Exercise 3.14 in~\cite{conforti_integer_2014}). Hence, if we denote by $A^j$ the $j$-th column of $A$, we have that
    \begin{align*}
        &\left( \sum_{j \in J_{1}} \alpha_{j} - \beta \right)
          f_{k}(J_{0}, J_{1})
        + \sum_{j \in [n] \setminus S}
          \alpha_j\, f_{k+1}(J_{1} \cup \{j\}, J_{0})
        \\
        &=  \left( \sum_{j \in J_{1}} (u^\top A^j) - u^\top b \right)
          f_{k}(J_{0}, J_{1})
        + \sum_{j \in [n] \setminus S}
          (u^\top A^j)\, f_{k+1}(J_{1} \cup \{j\}, J_{0})\\
          &=  u^{\top}\left(\left( \sum_{j \in J_{1}} A^{j} - b \right)
          f_{k}(J_{0}, J_{1})
        + \sum_{j \in [n] \setminus S}
          A^{j}\, f_{k+1}(J_{1} \cup \{j\}, J_{0})\right),
    \end{align*}
    which is a conic combinations of inequalities in $\M^k(\{Ax\ge b\})$.
\end{proof}

\subsection{Proof of Proposition~\ref{proposition:symmetry-is-inherited}}\label{appendix:proof_symmetry-is-inherited}

Assume that $P := \{ x \in \R^{n} \, : \, a_{i}^\top x \geq b_{i}, i \in [m]\} \subseteq [0,1]^{n}$, and recall that we are assumomg that the description of $P$ contain explicitly the linear inequalities $0\le x\le \ones$. Due to Lemma~\ref{lm:CharacterizationG(P)}, if $\pi \in G(P)$ we have that $(\pi a_{i})^\top x \geq b_{i}$ is valid for $P$. Hence, by Lemma~\ref{lemma:SAaddIneq}, if we add $(\pi a_{i})^\top x \geq b_{i}$ to the description of $P$, the set $\M^k(P)$ does not change. Therefore, we can assume that the set of inequalities describing $P$, $\{a_{i}^\top x \geq b_{i}: i \in [m]\}$, is close under permutations in $G(P)$, in the sense that $\alpha^\top x\ge \beta$ belongs to $\{a_{i}^\top x \geq b_{i}: i \in [m]\}$ if and only if $(\pi \alpha)^\top x\ge \beta$ belongs to $\{a_{i}^\top x \geq b_{i}: i \in [m]\}$.

We begin by introducing notation that will be useful for representing a description of $\M^{k}(P)$ by inequalities, see e.g.~\cite{sherali+adams94}. Let $x \in [0,1]^{n}$ and $y \in [0,1]^{\binom{n}{\leq k+1}}$ be two indeterminates. Given a subset $S \subseteq [n]$, $|S| = k$ and $2$-partition $J_{0} \sqcup J_{1} = S$, we define the multivariate polynomial $F_{k}(J_{0},J_{1})(x) := \prod_{i \in J_{1}} x_{i} \prod_{j \in J_{0}} (1 - x_{j})$. We denote its multilinearization by $f_{k}(J_{0},J_{1})(y)$. It is routine to check that $f_{k}(J_{0},J_{1})(y) = \sum_{H \subseteq J_{0}} (-1)^{|H|} y_{J_{1} \cup H}$, where $y_{I} := \prod_{i \in I} x_{i}$ for all $I \subseteq [n]$. By the latter, an explicit description of $\M^{k}(P)$ by inequalities is given by
\begin{subequations}\label{propostion:Mk_ineqs}
\begin{itemize}[leftmargin=3.5em,itemsep=1em,topsep=0.5em]
    \item[\textbf{SA-0:}] 
    \begin{align}
        y_{\emptyset} &= 1,
    \end{align}
    \item[\textbf{SA-1:}] $\forall J_{0} \sqcup J_{1} = S,\;
          S \subseteq [n],\;
          |S| = \min\{k+1, n\}$
    \begin{align}
        f_{|S|}(J_{0}, J_{1})(y) &\ge 0, \label{proposition:Mk_ineq1}
    \end{align}
    \item[\textbf{SA-2:}] $\forall J_{0} \sqcup J_{1} = S, \forall S \subseteq [n], |S| =k$
    \begin{align}
        \left( \sum_{j \in J_{1}} a_{i,j} - b_{i} \right)
          f_{k}(J_{0}, J_{1})
        + \sum_{j \in [n] \setminus S}
          a_{i,j}\, f_{k+1}(J_{1} \cup \{j\}, J_{0})
        &\ge 0, \label{proposition:Mk_ineq2}
    \end{align}
\end{itemize}
\end{subequations}
\noindent where $a_i=(a_{i,j})_{j \in [n]}$ denotes the $i$th row of the description of $P$ by inequalities. For $J_{0} \sqcup J_{1} = S$ with $|S| = k$, $J_{0}' \sqcup J_{1}' = S'$ with $|S'| = k+1 \leq n$, and $i \in [m]$ we define the vectors $a^{i}(J_{0},J_{1}), \, h(J_{0}',J_{1}') \in \R^{\binom{n}{\leq k+1}}$ as follows: for every subset $I \subseteq [n]$ with $|I| \leq k+1$,
{\small{
\begin{align*}
    r^{J_{0},J_{1}}(a_i,b_i)_{I} :&= 
    \begin{cases}
        (-1)^{|H|} \left( \sum_{j \in J_{1}} a_{i,j} - b_{i} \right) & \mbox{if $I = J_{1} \sqcup H, \, H \subseteq J_{0}$,} \\
        (-1)^{|H|} a_{i,j} & \mbox{if $I = J_{1} \sqcup \{j\} \sqcup H, \, H \subseteq J_{0}, \, j \in [n] \setminus S$,} \\
        0 & \mbox{o.w.}
    \end{cases} \\
    h^{J_{0}',J_{1}'}_{I} :&= 
    \begin{cases}
        (-1)^{|H|} & \mbox{if $I = J_{1}' \sqcup H$, $H \subseteq J_{0}'$,} \\
        0 & \mbox{o.w.}
    \end{cases}
\end{align*}
}}\begin{lemma}\label{lemma:action_on_constraints}
    Suppose that $G$ acts on $\R^{\binom{n}{\leq k+1}}$ as in~\eqref{action_on_R^2^n}. Let $a, y \in \R^{\binom{n}{\leq k+1}}$ and $\pi \in G$. Then $a^\top \pi y \geq 0$ if and only if $(\pi^{-1}a)^\top y \geq 0$.
\end{lemma}
\begin{proof}
    It is not hard to see that each $\pi \in G$ can be associated with a permutation matrix $R_{\pi} \in \{0,1\}^{\binom{n}{\leq k+1} \times \binom{n}{\leq k+1}}$, hence we can write the action of $\pi$ on $y$, i.e., $\pi y$, as $R_{\pi} y$. Moreover, a permutation matrix is an orthogonal transformation, i.e., $(R_{\pi})^{-1} = R_{\pi}^{\top}$, then we can write $a^\top \pi y = a^\top R_{\pi} y = (a^\top R_{\pi}) y = (R_{\pi}^\top a)^\top y = (R_{\pi}^{-1}a)^\top y = (\pi^{-1} a)^\top y$.
\end{proof}

Next, we prove that for any $y \in \M^{k}(P)$ and any $\pi \in G$, we have that $\pi y \in \M^{k}(P)$. Indeed, for $J_{0} \sqcup J_{1} = S$ with $|S| = k$ observe that $(\pi r^{J_{0},J_{1}}(a_i,b_i))_{I} = (r^{J_{0},J_{1}}(a_i,b_i))_{\pi^{-1}(I)} = (r^{\pi(J_{0}),\pi(J_{1})}(\pi^{-1}a_i,b_i))_{I}$ for $I \subseteq [n]$ with $|I| \leq k+1$.
Therefore, by Lemma~\ref{lemma:action_on_constraints} we have that 
\begin{align*}
    r^{J_{0},J_{1}}(a_i,b)^\top (\pi y) \geq 0 &\iff (\pi^{-1}r^{J_{0},J_{1}}(a_i,b))^\top y \geq 0 \\
    &\iff (r^{\pi^{-1}(J_{0}),\pi^{-1}(J_{1})}(\pi^{-1}a_i,b_i))^\top y \geq  0.
\end{align*}
Since $\pi$ is a permutation action on $[n]$, we have that $\pi (S) \subseteq [n]$, $|\pi (S) | = k$ and $\pi (J_{0}) \sqcup \pi(J_{1}) = \pi(S)$ is a $2$-partition. Finally, as we assume that if the inequality $a_i^\top x\ge b$ is present in our description of $P$, then also $(\pi^{-1} a_i)^\top x\ge b$ is also present. Hence $(r^{\pi^{-1}(J_{0}),\pi^{-1}(J_{1})}(\pi^{-1}a_i,b_i))^\top y \geq  0$ is a valid inequality of $\M^k(P)$. The same argument can be applied to show that $h^{J_{0}',J_{1}'} (\pi y) \geq 0$ holds for an arbitrary $2$-partition $J_{0}' \sqcup J_{1}' = S'$ with $|S'| = k+1 \leq  n$.
We conclude that if $y\in \M^k(P)$ then also $\pi y\in \M^k(P)$. Thus $\M^k(P)$ is $G$-invariant.

\subsection{Symmetric Operators}\label{app:symmetricOps}

In this section, we give a proof of Proposition~\ref{proposition:symmetry_preserving}.
Let $\mathcal{C}_n$ be the collection of all polytopes in $\mathbb{R}^n$. We say that an operator $T: \mathcal{C}_{n} \to \mathcal{C}_{n}$ is \emph{well-defined} if for every polytope $P \subseteq [0,1]^{n}$: (1) $T$ is integer valid, i.e., $T(P) \cap \{0,1\}^{n} = P \cap \{0,1\}^{n}$, and (2) and $T(P) \subseteq P$; and (3) $T$ is monotone, i.e., for every pair of polytopes $P,Q \subseteq [0,1]$, if $P\subseteq Q$, then $T(P) \subseteq T(Q)$. We also say that $T$ is \emph{symmetry-preserving} if for every $G$-invariant polytope $P$, $T(P)$ is also $G$-invariant. Finally, for a permutation group $G\le \mathcal{S}_n$, we define its fixed space as $\fix(G):=\{x\in \mathbb{R}^n: \pi x=x \text{ for all } \pi\in G\}$.

\begin{proposition}\label{proposition:symmetric_operator}
    Let $P$ be a $G$-invariant polytope. If $T$ is a well-defined symmetry-preserving operator, then 
    \begin{equation*}
        T(P) \neq \emptyset \iff \fix(G) \cap T(P) \neq \emptyset.
    \end{equation*}
\end{proposition}
\begin{proof}
    Necessity is straightforward. To prove sufficiency, pick any $x \in T(P)$. Given that $T$ is symmetry-preserving we have that $\orb_{G}(x)\subseteq T(P)$. By convexity of $T(P)$ we have that $\conv (\orb_{G}(x)) \subseteq T(P)$. Finally, the barycenter of the orbit of $x$ under $G$, $\beta_G(x)$, clearly belongs to $\conv(\orb_{G}(x))\subseteq T(P)$. Finally, we observe that $\beta_G(x)\in \fix(G)$, as for every $\pi'\in G$,
    \[
    \pi' \beta_G(x)
= \pi'\frac1{|G|}\sum_{\pi\in G} \pi x= \frac1{|G|}\sum_{\pi\in G} \pi'\pi x = \frac1{|G|}\sum_{\pi''\in G} \pi'' x,\]
where in the last equality we used the changed of variables $\pi'':=\pi'\pi$.
\end{proof}

Let $Y \in \R^{(n+1) \times (n+1)}$ be a matrix with row and column indices in $\{0,1,\ldots,n\}$. For a permutation $\pi\in \mathcal{S}_n$, we extend it to $\{0,1\ldots,n\}$ by simply defining $\pi(0)=0$. We define the \emph{action of $G$ on any lifting matrix}  as $(\pi \cdot Y)_{i,j} := Y_{\pi^{-1}(i),\pi^{-1}(j)}$ for every $i,j \in \{0,1,\ldots,n\}$.

\begin{proof}[Proof of Proposition~\ref{proposition:symmetry_preserving}]
    It suffices to prove the statement for $k=1$. Let us show first that $\LS_{0}$ is symmetry-preserving. Assume that $P$ is $G$-invariant. We aim to show that $\LS_{0}(P)$ is also $G$-invariant, that is, if $x \in \LS_{0}(P)$ and $\pi\in G$ then $\pi x \in \LS_{0}(P)$. Let $Y$ be the lifting matrix associated to $x$, where $Ye_{0} = (1, x)^{\top}$ and $Ye_i,Y(e_0-e_i)\in K(P)$.
    We show that $\pi \cdot Y$ is a lifting matrix for $\pi x$. Let us extend the action of $\pi$ to vectors in $\mathbb{R}^{n+1}$ as $\pi (x_0,\ldots,x_n)= (x_0,x_{\pi^{-1}(1)},\ldots,x_{\pi^{-1}(n)})$. Then, $(\pi\cdot Y)e_i = \pi\cdot (Y e_{\pi^{-1}(i)})\in \pi K(P) = \{\pi x: x\in K(P)\}$. Using the straightforward fact that $K(P)$ is $G$-invariant (for the action of $G$ extended to $\mathbb{R}^{n+1}$), we obtain that $\pi K(P) = K(P)$, and hence $(Y e_{\pi^{-1}(i)})\in K(P)$ for every $i\in [n]$. With this we conclude that $(\pi Y) e_{i}\in K(P)$. The exact same analysis shows that  $(\pi Y)(e_0-e_{i})\in K(P).$ Finally, we observe that due to the definition of the action of $\pi$ on $Y$ it holds that $(\pi Y)e_0 = (1,\pi x)= (\pi Y)^{\top}e_0=\text{diag}(\pi Y)$. We conclude that $\pi x \in \LS_0$. 

    Finally, assume that $x\in \LS(P)$. We follow the same proof as for $\LS_0$ and show that if $Y$ is the lifting matrix for $x\in \LS$, then $\pi Y$ is the corresponding lifting matrix for $\pi x$. Indeed, we have that $Y$ is symmetric, $Y=Y^{\top}$. By using the definition of $\pi Y$, we obtain that $\pi Y$ is also symmetric, as for all $i,j\in \{0,1\ldots,n\}$, 
    \[(\pi Y)_{i,j}= Y_{\pi^{-1}(i),\pi^{-1}(j)}=Y_{\pi^{-1}(j),\pi^{-1}(i)} = (\pi Y)_{j,i}.\]
    We conclude that $\pi x\in \LS(P)$.   
\end{proof}

\section{Proof of Theorem~\ref{main_theorem}}\label{app:proof-main-theorem}

\subsection{\texorpdfstring{Proofs from~\ref{mainthm-a} $\Rightarrow$ \ref{mainthm-b}}{Proofs from \ref{mainthm-a} ⇒ \ref{mainthm-b}}}

We present here the proofs from the deduction~\ref{mainthm-a}$\Rightarrow$\ref{mainthm-b} that are excluded from the main body. These are the proof of Lemma~\ref{cl:system-in-gamma}, and the proofs of Steps 3.~and 4. We start with the former.

\begin{proof}[Proof of Lemma~\ref{cl:system-in-gamma}]
Let $\overline \gamma, \overline \lambda$ be a solution to the system~\eqref{eq:z_in_M0-new}--\eqref{eq:lambda-gamma-nonnegative} that is strictly positive in all components and set $\overline y_I = \overline \gamma_{|I|}$ for all $I\subseteq [n]$, $|I|\leq k+1$.
Observe that $\overline y \in [0,1]^{\binom{n}{\leq k+1}}$. 

For a $2$-partition $J_0 \sqcup J_1$ of $S$ with $|S|=k$, let $\overline{y}^{J_0,J_1}$ and $z^{J_{0}, J_{1}}(\overline y)$ be as in~\eqref{eq:z}. To show that $\overline y \in \M^{k}(P)$, it suffices to verify that the conditions from Proposition~\ref{prop:SA_characterization} hold. In other words, (i) $z^{J_{0},J_{1}}(\overline{y}) \in \M^{0}(P)$ for every $2$-partition $J_{0} \sqcup J_{1} = S$, and (ii) $\overline y_{I} = \sum_{J_{0} \sqcup J_{1} = S} \overline y^{J_0,J_1}_\emptyset z^{J_{0},J_{1}}_{I}(\overline y)$ for every $I \subseteq [n]$ with $|I| \leq 1$. 

Observe that, for $|J_1|=\ell$, we have \begin{equation}\label{eq:useful}\overline y^{J_0,J_1}_\emptyset= \sum_{H \subseteq J_0} (-1)^{|H|} \overline{y}_{J_1 \cup H} = \sum_{r =0}^{k-\ell} (-1)^{r} \binom{k-\ell}{r} \overline \gamma_{r + \ell} = \overline \lambda_{\ell} > 0,\end{equation}
where the first equality holds by definition of $y^{J_0,J_1}_\emptyset$, the second by definition of $\overline y$, the third by~\eqref{eq:ugly-inequality}, and the inequality by the fact that $\overline \lambda_\ell>0$ by hypothesis.
Thus, $z^{J_0,J_1}_I(\overline y)=\overline y^{J_{0},J_{1}}_{I}/\overline y^{J_{0},J_{1}}_{\emptyset}$ for all $I\subseteq [n]$, $|I|\leq 1$.

We first show that condition (i) holds for $2$-partitions of the form $J_{0} = [k] \setminus [\ell]$, $J_{1} = [\ell]$ for $\ell \in \{0,\dots,k\}$. We abbreviate $z^{\ell}(\overline{y}) := z^{[k] \setminus [\ell],[\ell]}(\overline y)$ and $\overline y^{\ell}:=\overline y^{[k]\setminus [\ell],[\ell]}$.

We show (i) by verifying that $z^{\ell}(\overline{y}) = (1, \overline x^\ell )$ for every $\ell\in \{0,\dots,k\}$. The statement follows by Lemma~\ref{cl:z-in-M0}. 
Let $I \subseteq [n]$ with $|I| \leq 1$. For $I = \emptyset$, then ${z}_{I}^{\ell}(\overline y) = 1$ by definition. Thus, let $I=\{j\}$. To compute $z^\ell_j(\overline y)$, observe that for $j \in [\ell]$ we can write
\begin{align*}\overline y^{\ell}_{j} & = \sum_{H \subseteq [k]\setminus [\ell]} (-1)^{|H|} \overline y_{[\ell] \cup \{j\} \cup H} = \sum_{H \subseteq [k]\setminus [\ell]} (-1)^{|H|} \overline y_{[\ell] \cup H} = \overline y^\ell_\emptyset,
\end{align*}
where the last equality follows by definition of $\overline y_\emptyset^{\ell}$ in \eqref{eq:z}.
Thus, $z_j^\ell(\overline y)=\overline y^\ell_\emptyset /\overline y^\ell_\emptyset=1=\overline x ^\ell_j$, as required.
Suppose now $j \in [k] \setminus [\ell]$. Then $j \in J_0$. Since we know that $\overline y \in [0,1]^{ \binom{n}{\leq k+1}}$, we can apply Lemma~\ref{lem:cancels-out} and obtain
\begin{align*}\overline y_{j}^{J_{0},J_{1}} & = \sum_{H \subseteq J_{0}} (-1)^{|H|} \overline y_{J_{1} \cup \{j\} \cup H} = 0.
\end{align*}
 Hence, we conclude that $z_j^\ell(\overline y)=0=\overline x^\ell_j$, as required.

Last, let $j \notin [k]$. We want to show that $\overline z_j^\ell=\overline x_j^\ell$. Notice that $\overline z_j^\ell=\overline y^\ell_j/\overline y^\ell_\emptyset=\overline y^\ell_j/\overline \lambda_\ell$, where the last equation follows  by~\eqref{eq:useful}. In turn, $\overline x_j^\ell=\Delta_\ell$ by definition. Since
\begin{align*}\overline y^{\ell}_{j} & = \sum_{H \subseteq [k]\setminus [\ell]} (-1)^{|H|} \overline y_{[\ell] \cup \{j\} \cup H} = \sum_{r =0}^{k-\ell} (-1)^{r}\binom{k-\ell}{r} \overline \gamma_{r + \ell+1},
\end{align*} 
we have that $\frac{\overline{y}_{j}^{\ell}}{\overline{\lambda}_{\ell}} = \overline x^{\ell}_j$ is equivalent to 
\begin{align}
     &\sum_{r=0}^{k-\ell} \binom{k-\ell}{r} (-1)^{r} \overline{\gamma}_{\ell+1+r} = \Delta_{\ell} \sum_{r=0}^{k-\ell} \binom{k-\ell}{r} (-1)^{r} \overline{\gamma}_{\ell+r} \nonumber\\
    \iff &\sum_{r=0}^{k-\ell} \binom{k-\ell}{r} (-1)^{r} \overline{\gamma}_{\ell+1+r} + \sum_{r=1}^{k-\ell} \binom{k-\ell}{r}\Delta_{\ell} (-1)^{r+1} \overline{\gamma}_{\ell+r} = \Delta_{\ell} \overline{\gamma}_{\ell}  \nonumber \\
    \iff &\sum_{r=0}^{k-\ell} (-1)^{r} \left[ \binom{k-\ell}{r} + \binom{k-\ell}{r+1} \Delta_{\ell} \right] \overline{\gamma}_{\ell+1+r} = \Delta_{\ell} \overline{\gamma}_{\ell} \nonumber\\
    \iff &-\Delta_{\ell} \overline{\gamma}_{\ell}+ \sum_{r=\ell+1}^{k+1} (-1)^{\ell + r - 1} \left[ \binom{k-\ell}{r- \ell - 1} + \binom{k-\ell}{r-\ell} \Delta_{\ell} \right] \overline{\gamma}_{r} = 0\label{linearsystem:conditioning2}
\end{align}
where the penultimate equivalence follows by a change of index variables $r' = r+1$ in the second sum at the LHS. Equation \eqref{linearsystem:conditioning2} is then implied by~\eqref{eq:z_in_M0-new},~\eqref{eq:ugly-inequality},~\eqref{linearsystem:A_definition-correct} and $\overline{\gamma}_{0} = 1$. This concludes the proof that $\overline z^\ell = (1,\overline{x}^{\ell})$. Since $P$ is $G$-invariant, we deduce that $\overline z^{J_0,J_1} = (1,x^{J_0,J_1})$ for all $S\subseteq [n]$, $|S|=k$, and $J_0 \sqcup J_1 = S$. Hence, (i) holds. 

To show (ii), recall the we observed in~\eqref{eq:useful} that $\overline y_\emptyset^{J_0,J_1}=\overline \lambda_{|J_1|}$ for every $2$-partition $J_0 \sqcup J_1$ of $S\subseteq [n]$ with $|S|=k$. Let $I=\{j\} \in [n]$. Recall that we also just observed that $z_j^{J_0,J_1}(\overline y)=\overline x_j^{J_0,J_1}$. If $j \notin S$, then
\begin{align*} 
        \sum_{\substack{J_{0} \sqcup J_{1} = S}}
        \overline y_{\emptyset}^{J_{0},J_{1}} z_{j}^{J_{0},J_{1}}(\overline y) & = 
        \sum_{\substack{J_{0} \sqcup J_{1} = S}}\overline y_{\emptyset}^{J_{0},J_{1}} \overline x_j^{J_0,J_1}  \\
        & = \sum_{\substack{J_{0} \sqcup J_{1} = S}}\overline\lambda_{|J_1|} \Delta_{|J_1|} & \hbox{by~\eqref{eq:xell} and~\eqref{eq:useful}} \\
        & = \sum_{\ell=0}^{k} \binom{k}{\ell} \overline\lambda_{\ell} \Delta_{\ell}\\
        & = \overline \gamma_1 = \overline y_j & \hbox{by~\eqref{eq:convex_combination_singletons:fractional-new} and ~\eqref{eq:overliney}.}          
\end{align*}
If conversely $j \in S$, we have
\begin{align*}
        \sum_{\substack{J_{0} \sqcup J_{1} = S}}
        \overline y_{\emptyset}^{J_{0},J_{1}} z_{j}^{J_{0},J_{1}}(\overline y) & = 
        \sum_{\substack{J_{0} \sqcup J_{1} = S}}\overline y_{\emptyset}^{J_{0},J_{1}} \overline x_j^{J_0,J_1} \\
        & = 
        \sum_{\substack{J_{0} \sqcup J_{1} = S}: j \in J_0}\overline y_{\emptyset}^{J_{0},J_{1}} \overline x_j^{J_0,J_1} + &\sum_{\substack{J_{0} \sqcup J_{1} = S}: j \in J_1}\overline y_{\emptyset}^{J_{0},J_{1}} \overline x_j^{J_0,J_1}\\
        & = \sum_{\substack{J_{0} \sqcup J_{1} = S}: j \in J_1}\overline\lambda_{|J_1|} &\hbox{by~\eqref{eq:xell}}\\
        & = \sum_{\ell=1}^{k} \binom{k-1}{\ell-1} \overline\lambda_{\ell}\\
          & = \overline \gamma_1 = \overline y_1 & \hbox{by~\eqref{eq:convex_combination_singletons:integer-new} and~\eqref{eq:overliney}.}
                \end{align*}
To conclude that (ii) holds, let $I=\emptyset$ and observe 
\begin{align*}\sum_{\substack{J_{0} \sqcup J_{1} = S}}
y_{\emptyset}^{J_{0},J_{1}} z_{\emptyset}^{J_{0},J_{1}}(y) & = \sum_{\substack{J_{0} \sqcup J_{1} = S}} \overline y_{\emptyset}^{J_{0},J_{1}}, \\ 
& = \sum_{\substack{J_{0} \sqcup J_{1} = S}} \overline \lambda_{|J_1|} & \hbox{by~\eqref{eq:useful}} \\
& = \sum_{\ell=0}^{k} \binom{k}{\ell} \overline \lambda_\ell \\
& = 1 & \hbox{by~\eqref{eq:convex_combination_emptyset-new}.}
\end{align*}
\end{proof}

\paragraph{3.~Solution to the system defining $\overline \gamma$.} We are left to show the existence of a solution to the  system from Lemma~\ref{cl:system-in-gamma}. To achieve that, we give in Lemma~\ref{linearsystem:determinant_theorem} an explicit formula for the determinant of $A^{t}_{\Theta}$. Recall that this is the constraint matrix of system~\eqref{eq:z_in_M0-new} that uniquely defines $\overline \gamma$. To prove this formula, we start with some intermediate steps. For a vector $\Theta=(\Theta_0,\dots,\Theta_k)\in (0,1)^{t+1}$, we sometimes make the dependency on the components of $\Theta$ more explicit and write $A^t_{\Theta_0,\dots,\Theta_{t+1}}=A^t_{\Theta}$.

\begin{lemma}\label{lem:decompose-detA}
Let $t \in \N$ and $\Theta=(\Theta_0,\dots, \Theta_1) \in (0,1)^{t+1}$. The determinant of $A^{t}_{\Theta}$ can be decomposed as
\begin{equation}\label{linearsystem:determinant_decomposition}
        \det A^{t}_{\Theta} = \Theta_{0} \det A^{t-1}_{\Theta_{1},\dots,\Theta_{t}} + (1-\Theta_{t}) \det A^{t-1}_{\Theta_{0},\dots,\Theta_{t-1}}.
    \end{equation}
    \end{lemma}
{

        The idea of the proof of Lemma~\ref{lem:decompose-detA} is to transform $A^{t}_{\Theta}$ via elementary column operations into another matrix $B^{t}_{\Theta}$ for which we deduce the relation~\eqref{linearsystem:determinant_decomposition}. As $\det A^{t'}_{\Theta'} = \det B^{t'}_{\Theta'}$ for all $t',\Theta'$, the relation~\eqref{linearsystem:determinant_decomposition} also holds for $A^{t}_{\Theta}$.      
}
For $t \in \N$ and $\Theta=(\Theta_0,\dots,\Theta_t) \in \R^t$, we define 
$B^{t}_{\Theta_{0},\dots,\Theta_{t}} \in \R^{(t+1)\times(t+1)}$ with rows indexed in $\{0,\dots,t\}$ and columns indexed in $\{1,\dots,t+1\}$ as follows: its $r$-th column is given by
        \begin{equation*}
            (B^{t}_{\Theta})_{\cdot,r} := \sum_{j=r}^{t} (A^{t}_{\Theta})_{\cdot,j}
        \end{equation*}

\begin{example}\label{ex:matrix-A}We next illustrate matrices $A^t_\Theta$ and $B^t_\Theta$ with some examples. We can write $A^1_{\Theta_0,\Theta_1}$ and transform it into $B^1_{\Theta_0,\Theta_1}$ as follows:  
        \begin{equation*}
        A^1_{\Theta_0,\Theta_1}=\begin{pmatrix}
        1 + \Theta_{0} & - 1 \\
        - \Theta_{1} & 1 
        \end{pmatrix}
        \underset{\text{\footnotesize adding 2 to 1}}{\rightsquigarrow}
        \begin{pmatrix}
            \Theta_{0} & - 1 \\
            1-\Theta_{1} & 1
        \end{pmatrix} = B_{\Theta_{0},\Theta_{1}}^{1}.
        \end{equation*}
Similarly, letting $\Theta=(\Theta_0,\Theta_1,\Theta_2)$, we have, 
 {\small{       \begin{align*}
A_{\Theta}^{2} =         \begin{pmatrix}
        1 + 2 \Theta_{0} & -2 - \Theta_{0} & 1 \\
        - \Theta_{1} & 1 + \Theta_{1} & - 1 \\
        0 & - \Theta_{2} & 1 
        \end{pmatrix}
        &\underset{\text{\footnotesize adding 2,3 to 1}}{\rightsquigarrow}
        \begin{pmatrix}
            \Theta_{0} & -2 - \Theta_{0} & 1 \\
            0 & 1 + \Theta_{1} & - 1 \\
            1-\Theta_{2} & - \Theta_{2} & 1 
        \end{pmatrix} \nonumber \\
        &\underset{\text{\footnotesize adding 3 to 2}}{\rightsquigarrow}
        \begin{pmatrix}
            \Theta_{0} & -1 - \Theta_{0} & 1 \\
            0 & \Theta_{1} & - 1 \\
            1-\Theta_{2} & 1 - \Theta_{2} & 1 
        \end{pmatrix}
        =B^{2}_{\Theta}.
        \end{align*}}}
\end{example}

\begin{myclaim}\label{cl:B}    For $\ell \in \{0,\dots,t\}$ and $r \in \{1,\dots,t+1\}$, entry $(\ell,r)$ of $B^{t}_{\Delta}$ is given by  \begin{equation}\label{full_dimensional:definition_B-correct}
    (B^{t}_{\Delta})_{\ell,r} = 
    \begin{cases}
        (-1)^{\ell + r - 1} \left[ \binom{t-\ell - 1}{r - \ell - 2} + \binom{t-\ell - 1}{r - \ell - 1} \Theta_{\ell} \right] & \text{if $\ell < r \leq t$}, \\
       1-\Theta_{t} & \text{if $\ell =t$ and $r \leq t$}, \\
       (-1)^{\ell + t} & \text{if $\ell \leq t$ and $r = t + 1$}, \\
       0  & \text{o.w.~(if $\ell \geq r$, $\ell <t$)}
    \end{cases}
    \end{equation}
\end{myclaim}
    \begin{proof}
    We will use the following binomial identities (see e.g.~\cite{graham+etal94}):
    \begin{align}
        \sum_{j=0}^{t} (-1)^{j} \binom{t}{j} &= 0, \label{full_dimensional:bin_identity1}\\
        \sum_{j=r}^{t} (-1)^{j} \binom{t}{j} &= (-1)^{r} \binom{t-1}{r-1}. \label{full_dimensional:bin_identity2}
    \end{align}

    \paragraph{Row $\ell = t$ and column $r = t+1$.}
    By definition $(B^{t}_{\Theta})_{\cdot,t+1} = (A^{t}_{\Theta})_{\cdot,t+1}$, and this agrees with formula~\eqref{full_dimensional:definition_B-correct} since $(A^{t}_{\Theta})_{\ell,t+1} = (-1)^{\ell+t}$ for $\ell \in \{0,\dots,t\}$. Next, we check the entries of row $\ell = t$ of $B^{t}_{\Theta}$ for columns $r \in \{0,\dots, t\}$. Since the only nonzeros in row $\ell = t$ of $A^{t}_{\Theta}$ are $-\Theta_{t}$ and $(-1)^{t+t}$, at entries $(t,t)$ and $(t,t+1)$, respectively, it holds that $(B^{t}_{\Theta})_{t,r} = 1 - \Theta_{t}$ for $r \in \{1,\dots,t\}$.

    \paragraph{Lower triangular part of $B^{t}_{\Theta}$.} We check $(B^{t}_{\Theta})_{\ell,r} = 0$ for $1 \leq r \leq \ell < t$. Indeed, 
    \begin{align*} %= \sum_{j=r}^{t} (A^{t}_{\Theta})_{\ell,r}
        (B^{t}_{\Theta})_{\ell,r} &= - \Theta_{\ell} + \sum_{j=\ell+1}^{t+1} (-1)^{\ell + j - 1} \left[ \binom{t-\ell}{j-\ell-1} + \binom{t-\ell}{j-\ell} \Theta_{\ell} \right] \\
        &= - \Theta_{\ell} + \sum_{j=\ell+1}^{t+1} (-1)^{\ell + j - 1} \binom{t-\ell}{j-\ell-1} + \sum_{j=\ell+1}^{t} (-1)^{\ell + j - 1} \binom{t-\ell}{j-\ell} \Theta_{\ell}  \\
        &= - \Theta_{\ell} + \sum_{j=0}^{t-\ell} (-1)^{2\ell + j} \binom{t-\ell}{j} + \sum_{j=1}^{t-\ell} (-1)^{2\ell + j - 1} \binom{t-\ell}{j} \Theta_{\ell}  \\
        &= - \Theta_{\ell} + (-1)^{2\ell} \sum_{j=0}^{t-\ell} (-1)^{j} \binom{t-\ell}{j} + (-1)^{2\ell -1 }\sum_{j=1}^{t-\ell} (-1)^{j} \binom{t-\ell}{j} \Theta_{\ell}  \\
        &= - \Theta_{\ell} + (-1) \left(- \Theta_{\ell} + \sum_{j=0}^{t-\ell} (-1)^{j} \binom{t-\ell}{j} \Theta_{\ell}\right) = 0,
     \end{align*}
    where in the last two equalities we used the identity~\eqref{full_dimensional:bin_identity1}.

    \paragraph{Diagonal and upper triangular part of $B^{t}_{\Theta}$.} Let $0 \leq \ell < r \leq t$. Using~\eqref{full_dimensional:bin_identity2}, we have:
    \begin{align*}%= \sum_{j=r}^{t} (A^{t}_{\Theta})_{\ell,r}
        (B^{t}_{\Theta})_{\ell,r} &= \sum_{j=r}^{t+1} (-1)^{\ell + j-1} \left[ \binom{t-\ell}{j-\ell-1} + \binom{t-\ell}{j-\ell} \Theta_{\ell} \right] \\
        &= (-1)^{2\ell} \sum_{j=r-\ell-1}^{t-\ell} (-1)^{j}\binom{t-\ell}{j} + (-1)^{2\ell-1} \sum_{j=r-\ell}^{t-\ell} (-1)^{j} \binom{t-\ell}{j} \Theta_{\ell} \\
        % &= (-1)^{2\ell} \sum_{j=r-\ell}^{t-\ell} (-1)^{j}\binom{t-\ell}{j} + (-1)^{2\ell+1} \sum_{j=r-\ell+1}^{t-\ell} (-1)^{j} \binom{t-\ell}{j} \Theta_{\ell} \\
        &= (-1)^{2\ell} (-1)^{r-\ell-1} \binom{t-\ell-1}{r-\ell-2} + (-1)^{2\ell -1 } (-1)^{r-\ell} \binom{t-\ell-1}{r-\ell-1} \Theta_{\ell} \\
        &= (-1)^{\ell+r-1} \left[ \binom{t-\ell-1}{r-\ell-2} + \binom{t-\ell-1}{r-\ell-1} \Theta_{\ell} \right].
    \end{align*}
    \end{proof}
\noindent Notice, in particular, that $(B^{t}_{\Theta})_{\ell,\ell+1} = (-1)^{2\ell} \left[ \binom{t-\ell-1}{-1} + \binom{t-\ell-1}{0} \Theta_{\ell} \right] = \Theta_{\ell}$.

\begin{proof}[Proof of Lemma~\ref{lem:decompose-detA}]    It suffices to prove~\eqref{linearsystem:determinant_decomposition} for $B^{t}_{\Theta}$. To this end, we expand the determinant along the first column of $B^{t}_{\Theta}$. This column has exactly two nonzeros, $\Theta_{0}$ and $1-\Theta_{t}$ in row $0$ and $t$ respectively, so we obtain
    \begin{equation*}
        \det B^{t}_{\Theta} = \Theta_{0} \det B^{t-1}_{\Theta_{1},\dots,\Theta_{t}} + (-1)^{t} (1-\Theta_{t}) \det C,
    \end{equation*}where $C$ is the matrix obtained by removing the first column and the last row of $B^t_\Theta$. We claim that $\det C= (-1)^{t} \det A^{t-1}_{\Theta_{0},\dots,\Theta_{t-1}}$. Indeed, observe that $C$ is the $t \times t$ matrix obtained by removing the first column and the last row of $B^{t}_\Theta$, and every nonzero entry in this submatrix has the opposite sign of the corresponding entries in $A^{t-1}_{\Theta_{0},\dots,\Theta_{t-1}}$.
    Since $\det A^{t-1}_{\Theta_{0},\dots,\Theta_{t-1}} = \det B^{t-1}_{\Theta_{0},\dots,\Theta_{t-1}}$, by the multilinearity of the determinant, we can conclude that \begin{equation*}
        \det B^{t}_{\Theta} = \Theta_{0} \det B^{t-1}_{\Theta_{1},\dots,\Theta_{t}} + (1-\Theta_{t}) \det B^{t-1}_{\Theta_{0},\dots,\Theta_{t-1}}.
    \end{equation*}
    \end{proof}

For $t \in \N$ and a vector $\Theta=(\Theta_0,\dots,\Theta_t) \in (0,1)^{t+1}$, we define
\begin{equation}\label{linearsystem:Dk_definition}
    \omega^{t}_{\Theta} := \sum_{r=0}^{t} \binom{t}{r} \prod_{j=0}^{r-1} \Theta_{j} \prod_{j=r+1}^{t}(1-\Theta_{j}),
\end{equation}
and observe that $\omega_\Theta^t$ is strictly positive. Equivalently, we write  $\omega^{t}_{\Theta}= \omega^{t}_{\Theta_{0},\dots,\Theta_{t}}$. 

\begin{lemma}\label{linearsystem:determinant_theorem}
    Let $t \in \N$ and $\Theta=(\Theta_0,\dots,\Theta_t) \in (0,1)^{t+1}$. The determinant of $A^{t}_{\Theta}$ is equal to $\omega^{t}_{\Theta}$. 
    \end{lemma}
\begin{proof}
We use the decomposition~\eqref{linearsystem:determinant_decomposition} and show the claim by induction on $t$. The case $t=1$ easily follows since
    \begin{align*}
        \omega^{1}_{\Theta_{0},\Theta_{1}} = (1-\Theta_{1}) + \Theta_{0} 
        = \det A^{1}_{\Theta_{0},\Theta_{1}}.
    \end{align*}
    Next, suppose that $\omega^{t-1}_{\Theta'} = \det A^{t-1}_{\Theta'}$ for $\Theta' \in (0,1)^{t}$ arbitrary. Using~\eqref{linearsystem:determinant_decomposition} and our induction hypothesis, we have
    
    \begin{align*}
        \det A^{t}_{\Theta} &=  \Theta_{0} \det A^{t-1}_{\Theta_{1},\dots,\Theta_{t}} + (1-\Theta_{t}) \det A^{t-1}_{\Theta_{0},\dots,\Theta_{t-1}} \\
        &= \Theta_{0} \omega^{t-1}_{\Theta_{1},\dots,\Theta_{t}} + (1-\Theta_{t}) \omega^{t-1}_{\Theta_{0},\dots,\Theta_{t-1}} \\
        &= \Theta_{0} \omega^{t-1}_{\Theta_{1},\dots,\Theta_{t}} + (1-\Theta_{t}) \left( \prod_{j=1}^{t-1} (1- \Theta_{j}) + \sum_{r=1}^{t-1} \binom{t-1}{r} \prod_{j=0}^{r-1} \Theta_{j} \prod_{j=r+1}^{t-1} (1-\Theta_{j}) \right) \\
        &= \Theta_{0} \omega^{t-1}_{\Theta_{1},\dots,\Theta_{t}} + \prod_{j=1}^{t} (1-\Theta_{j}) + \Theta_{0} \sum_{r=1}^{t-1} \binom{t-1}{r} \prod_{j=1}^{r-1} \Theta_{j} \prod_{j=r+1}^{t} (1-\Theta_{j}) \\
        &= \prod_{j=1}^{t} (1-\Theta_{j}) + \Theta_{0} \left( \sum_{r=1}^{t-1} \left[ \binom{t-1}{r-1} + \binom{t-1}{r} \right] \prod_{j=1}^{r-1} \Theta_{j} \prod_{j=r+1}^{t} (1-\Theta_{j}) + \binom{t-1}{t-1} \prod_{j=1}^{t-1} \Theta_{j} \right) \\
        (*) &= \prod_{j=1}^{t} (1-\Theta_{j}) + \Theta_{0} \left( \sum_{r=1}^{t} \binom{t}{r} \prod_{j=1}^{r-1} \Theta_{j} \prod_{j=r+1}^{t} (1-\Theta_{j}) \right) \\
        &= \sum_{r=0}^{t} \binom{t}{r} \prod_{j=0}^{r-1} \Theta_{j} \prod_{j=r+1}^{t} (1-\Theta_{j}) \\
        &= \omega^{t}_{\Theta},
    \end{align*}
    where in $(*)$ we used $\binom{t-1}{r-1} + \binom{t-1}{r} = \binom{t}{r}$. 
\end{proof}

   \paragraph{4.~Solution to the general system.} We next define a vector $(\overline \gamma, \overline \lambda) \in \R^{(k+1)+(k+1)}$ and show that satisfies all constraints from Lemma~\ref{cl:system-in-gamma}. 
   
   For $\ell \in \{0,\dots,k\}$, define
\begin{align}\label{linearsystem:lambda_solution}
    \overline{\lambda}_\ell := \frac{1}{\omega^{k}_{\Delta}} \prod_{i=0}^{\ell-1} \Delta_i \prod_{i=\ell+1}^k (1-\Delta_i).
\end{align}
For $i \in [k+1]$, we define
    \begin{align}
        \rho^{i}_{\Delta} := \begin{cases}
            \Delta_{i-1} \left( \frac{\omega^{k-i}_{\Delta_{i},\dots,\Delta_{k}}}{\omega^{k-(i-1)}_{\Delta_{i-1},\dots,\Delta_{k}}} \right) & \text{if $i \leq k$}, \\
            \Delta_{k} & \text{if $i =k+1$.}
        \end{cases}
    \end{align}
   These scalars are well-defined since $\omega^{k-i}_{\Delta_{i},\dots,\Delta_{k}} >0$. Then define $\overline{\gamma} \in \R^{k+1}$ as
    \begin{equation}\label{linearsystem:gamma_solution}
        \overline{\gamma}_{i} := \rho_{\Delta}^{i} \overline{\gamma}_{i-1}
        % \begin{cases}
        %     \rho_{\Delta}^{i} \overline{\gamma}_{i-1}  & \text{if $i \leq k$}, \\
        %    \Delta_{k} \overline{\gamma}_{k}  & \text{if $i = k+1$},
        % \end{cases}
    \end{equation}
    for $i \in [k+1]$ and $\overline{\gamma}_{0} := 1$.

\smallskip 

\paragraph{Constraints~\eqref{eq:lambda-gamma-nonnegative}.} $\overline{\lambda} \in (0,1)^{k+1}$ immediately from $\Delta \in (0,1)^{k+1}$ and the definition of $\overline \lambda$. $\overline \gamma_0=1$ follows by definition. Last, we prove that $\overline{\gamma} \in (0,1]^{k+1}$. For $i \in \{0,\dots,k\}$ we can rewrite $\overline{\gamma}_{i}$ as
    \begin{equation}\label{linearsystem:gamma_solution2}
        \overline{\gamma}_{i} = \left( \prod_{j=0}^{i-1} \Delta_{j} \right) \left( \frac{\omega^{k-i}_{\Delta_{i},\dots,\Delta_{k}}}{\omega^{k}_{\Delta_{0},\dots,\Delta_{k}}} \right).
    \end{equation}

    Using~\eqref{linearsystem:Dk_definition}, we write 
    \begin{align*}
        \omega^{k-(i-1)}_{\Delta_{i-1},\dots,\Delta_{k}} 
        &= \sum_{r=0}^{k-(i-1)} \binom{k-(i-1)}{r} \prod_{j=0}^{r-1} \Delta_{j+i-1} \prod_{j = r+1}^{k} (1- \Delta_{j+i-1}) \\
        &= \sum_{r=i-1}^{k} \binom{k-(i-1)}{r-(i-1)} \prod_{j=i-1}^{r-1} \Delta_{j} \prod_{j = r+1}^{k} (1- \Delta_{j}) \\
        &= \prod_{j=i}^{k}(1-\Delta_{j}) + \sum_{r=i}^{k} \binom{k-(i-1)}{r-(i-1)} \prod_{j=i-1}^{r-1} \Delta_{j} \prod_{j = r+1}^{k} (1- \Delta_{j}) \\
        &= \prod_{j=i}^{k}(1-\Delta_{j}) + \sum_{r=i}^{k} \left( \binom{k-i}{r-i} + \binom{k-i}{r-(i-1)} \right) \prod_{j=i-1}^{r-1} \Delta_{j} \prod_{j = r+1}^{k} (1- \Delta_{j}) \\
        &= \sum_{r=i}^{k} \binom{k-i}{r-i} \prod_{j=i-1}^{r-1} \Delta_{j} \prod_{j = r+1}^{k} (1- \Delta_{j}) + E(\Delta) \\
        &= \Delta_{i-1} \sum_{r=i}^{k} \binom{k-i}{r-i} \prod_{j=i}^{r-1} \Delta_{j} \prod_{j = r+1}^{k} (1- \Delta_{j}) + E(\Delta) \\
                &= \Delta_{i-1} \omega^{k-i}_{\Delta_{i},\dots,\Delta_{k}}  + E(\Delta),
    \end{align*}
    where $E(\Delta)$ is strictly positive. 
    Thus, $0 < \Delta_{i-1} \omega^{k-i}_{\Delta_{i},\dots,\Delta_{k}} < \omega^{k-(i-1)}_{\Delta_{i-1},\dots,\Delta_{k}}$ and therefore $\rho_{\Delta}^{i} \in (0,1)$. Consequently, in view of~\eqref{linearsystem:gamma_solution}, and $\overline \gamma_0=1$, we deduce that $\overline{\gamma} \in (0,1)^{k+1}$. 

\newpage

\paragraph{Constraints~\eqref{eq:z_in_M0-new}.} We first show the following.
    \begin{lemma}\label{linearsystem:determinant_decomposition2}
    For $i \in \{0,\dots,k\}$, the determinant of $A^{k-i}_{\Delta_{i},\dots,\Delta_{k}}$ can be written as 
    \begin{equation}\label{linearsystem:determinant_decomposition2:formula}
        \sum_{r=i}^{k} (-1)^{r+i} \left[ \binom{k-i}{r-i} + \binom{k-i}{r-i+1} \Delta_{i} \right] \prod_{j=i}^{r-1} \Delta_{j+1} \det A^{k-(r+1)}_{\Delta_{r+1},\dots,\Delta_{k}}.
    \end{equation}
    where $A^{k-(r+1)}_{\Delta_{r+1},\dots,\Delta_{k}} := 1$ if $r = k$.
    \end{lemma}
    \begin{proof}
        Since $A^{k-i}_{\Delta_{i},\dots,\Delta_{k}}$ is the $(k+1-i) \times (k+1-i)$ principal submatrix of $A^{k}_{\Delta}$ indexed by rows $\{i,\dots,k\}$ and columns $\{i+1,\dots,k+1\}$, it suffices, without loss of generality, to consider the case $i=0$. Let $r \in \{1,\dots,k\}$ and let $\tilde{A}^{k-r}$ be the $(k + 1 - r) \times (k + 1 -r)$ submatrix of $A^{k}_{\Delta}$ obtained by deleting rows $1$ through $r$ and columns $1$ through $r$ (recall that $\{1,\dots,k+1\}$ is the set of column indices of $A^{k}_{\Delta}$ and $\{0,\dots,k\}$ is the set of row indices of $A^k_{\Delta}$). We claim that the matrix $\tilde{A}^{k-r}$ has only two nonzero entries in its first column (as does $A^{k}_{\Delta}$). Indeed, first, observe that $A^{k}_{\Delta}$ has the nonzero subdiagonal $(-\Delta_{1},\dots,-\Delta_{k})$ lying below its main diagonal. Moroeover, after deleting columns $1$ through $r$, the first column of the resulting submatrix contains $r+2$ consecutive nonzeros, followed by zeros. Removing rows $1$ through $r$ then leaves exactly two nonzero entries in the first column of $\tilde{A}^{k-r}$ (see Figure ~\ref{fig:tildeA}).
        % \vspace{-2em}
        \begin{figure}[H]
        \centering
        \[
        \left(
        \begin{array}{cc|ccc}
        \times & \times & \pmb{\times} & \pmb{\times} & \pmb{\times} \\ \hline
        \times & \times & \times & \times & \times \\
        0      & \times & \times & \times & \times \\ \hline
        0      & 0      & \pmb{\times} & \pmb{\times} & \pmb{\times} \\
        0      & 0      & \pmb{0}      & \pmb{\times} & \pmb{\times}
        \end{array}
        \right)
        \]
        \caption{Schematic representation of $A^{4}_{\Delta}$, 
        with the entries in bold corresponding to the surviving $\tilde{A}^{2}$ submatrix after deleting rows 1-2 and columns 1-2.}
        \label{fig:tildeA}
        \end{figure}
        Therefore, the Laplace expansion along the first column of $A^{k}_{\Delta}$ and $\tilde{A}^{k-r}$ yields:
        \begin{align}
            \det A^{k}_{\Delta} &= \left[ \binom{k}{0}+ \binom{k}{1} \Delta_{0} \right] \det A^{k-1}_{\Delta_{1},\dots,\Delta_{k}} + \Delta_{1} \det \tilde{A}^{k-1} \label{linearsystem:determinant_decomposition2:formula:detA}\\
            \det \tilde{A}^{k-r} &= (-1)^{r} \left[ \binom{k}{r} + \binom{k}{r+1} \Delta_{0} \right] \det A^{k-(r+1)}_{\Delta_{r+1},\dots,\Delta_{k}} \nonumber\\
            &+ \Delta_{r+1} \det \tilde{A}^{k-(r+1)}. \label{linearsystem:determinant_decomposition2:formula:det_tildeA}
        \end{align}
        By repeated substitution of $\det \tilde{A}^{k-r}$ in~\eqref{linearsystem:determinant_decomposition2:formula:detA}, the claimed formula follows.
    \end{proof}

        Let $i \in \{0,\dots,k\}$. Equate the formula for the determinant of $A^{k-i}_{\Delta_{i},\dots,\Delta_{k}}$ from  Lemma~\ref{linearsystem:determinant_theorem} with that from Lemma~\ref{linearsystem:determinant_decomposition2}. Then multiply both sides by $\prod_{j=0}^{i}\Delta_{i} / \omega^{k}_{\Delta}$ to  obtain
        \begin{align*}
            & \frac{\prod_{j=0}^{i}\Delta_{i}}{\omega^{k}_{\Delta}} \omega^{k-i}_{\Delta_{i},\dots,\Delta_{k}} = \\  & \frac{\prod_{j=0}^{i}\Delta_{i}}{\omega^{k}_{\Delta}} \sum_{r=i}^{k} (-1)^{r+i} \left[ \binom{k-i}{r-i} + \binom{k-i}{r-i+1} \Delta_{i} \right] \prod_{j=i}^{r-1} \Delta_{j+1} \omega^{k-(r+1)}_{\Delta_{r+1},\dots,\Delta_{k}}.
        \end{align*}
    By applying~\eqref{linearsystem:gamma_solution2} to both sides, we deduce,
        \begin{equation*}
            \Delta_{i} \overline{\gamma}_{i} = \sum_{r=i}^{k} (-1)^{r+i} \left[ \binom{k-i}{r-i} + \binom{k-i}{r-i+1} \Delta_{i} \right] \overline{\gamma}_{r+1},
        \end{equation*}
thus showing that the $i$-th constraint from~\eqref{eq:z_in_M0-new} is satisfied. 

\smallskip

    \paragraph{Constraints~\eqref{eq:convex_combination_singletons:fractional-new}.} We have \begin{align}
    & \omega^{k}_{\Delta} \sum_{r=0}^{k} \binom{k}{r} \overline \lambda_r \Delta_{r} \label{eq:start} \\
    & = \sum_{r=0}^{k} \binom{k}{r} \prod_{j=0}^{r} \Delta_{j} \prod_{j=r+1}^{k} (1-\Delta_{j}) \nonumber \\
      &= \Delta_{0} \sum_{r=0}^{k} \binom{k}{r} \prod_{j=0}^{r-1} \Delta_{j+1} \prod_{j=r+1}^{k} (1-\Delta_{j}) \nonumber \\ 
      & =\Delta_{0} \sum_{r=1}^{k} \binom{k}{r} \prod_{j=0}^{r-1} \Delta_{j+1} \prod_{j=r+1}^{k} (1-\Delta_{j}) + \Delta_{0} \binom{k}{0} \prod_{j=1}^{k} (1 -\Delta_{j}) + \Delta_{0} \binom{k}{k} \prod_{j=0}^{k-1} \Delta_{j+1} \nonumber \\ 
         & = \Delta_{0} \sum_{r=1}^{k-1} \left[ \binom{k-1}{r-1} + \binom{k-1}{r} \right] \prod_{j=0}^{r-1} \Delta_{j+1} \prod_{j=r+1}^{k} (1 - \Delta_{j}) + \Delta_{0} \binom{k}{0} \prod_{j=1}^{k} (1 -\Delta_{j}) \nonumber \\ 
        &\qquad + \Delta_{0} \binom{k}{k} \prod_{j=0}^{k-1} \Delta_{j+1} \nonumber\\
        &= \Delta_{0} \sum_{r=1}^{k-1} \binom{k-1}{r-1} \prod_{j=0}^{r-1} \Delta_{j+1} \prod_{j=r+1}^{k} (1 - \Delta_{j}) \label{full_dimensional:last2claims_1:1}\\
        &\qquad + \Delta_{0} \sum_{r=1}^{k-1} \binom{k-1}{r}  \prod_{j=0}^{r-1} \Delta_{j+1} \prod_{j=r+1}^{k} (1 - \Delta_{j}) \label{full_dimensional:last2claims_1:2}\\
        &\qquad + \Delta_{0} \binom{k-1}{0} \prod_{j=2}^{k}(1-\Delta_{j}) - \Delta_{0}\Delta_{1} \binom{k-1}{0} \prod_{j=2}^{k} (1-\Delta_{j}) \label{full_dimensional:last2claims_1:3}\\
        &\qquad + \Delta_{0} \binom{k-1}{k-1} \prod_{j=0}^{k-1} \Delta_{j+1}. \label{full_dimensional:last2claims_1:4}
    \end{align}
  
        \eqref{full_dimensional:last2claims_1:2} can be expanded as
    \begin{align}
        \Delta_{0} \sum_{r=1}^{k-1} &\left( \binom{k-1}{r} \prod_{j=0}^{r-1} \Delta_{j+1} \prod_{j=r+2}^{k} (1-\Delta_{j}) \right. \nonumber \\ 
        &\left.- \Delta_{r+1} \binom{k-1}{r} \prod_{j=0}^{r-1} \Delta_{j+1} \prod_{j=r+2}^{k}(1-\Delta_{j}) \right). \label{full_dimensional:last2claims_2}
    \end{align}

\noindent Then~\eqref{eq:start} equals~\eqref{full_dimensional:last2claims_1:3}$+$~\eqref{full_dimensional:last2claims_1:1}$+$~\eqref{full_dimensional:last2claims_2}$+$~\eqref{full_dimensional:last2claims_1:4}, which can be written as
{\small{
\begin{align}
        &\Delta_{0} \left( \underbrace{\binom{k-1}{0} \prod_{j=2}^{k}(1-\Delta_{j}) - \Delta_{1} \binom{k-1}{0} \prod_{j=2}^{k} (1-\Delta_{j})}_{\eqref{full_dimensional:last2claims_1:3}/\Delta_0} + \underbrace{\sum_{r=1}^{k-1} \binom{k-1}{r-1} \prod_{j=0}^{r-1} \Delta_{j+1} \prod_{j=r+1}^{k}(1-\Delta_{j})}_{\eqref{full_dimensional:last2claims_1:1}/\Delta_0} \right. \nonumber \\ & \left.+ \underbrace{\sum_{r=1}^{k-1} \binom{k-1}{r} \prod_{j=0}^{r-1} \Delta_{j+1} \prod_{j=r+2}^{k} (1-\Delta_{j})   -\sum_{r=1}^{k-1}  \Delta_{r+1} \binom{k-1}{r} \prod_{j=0}^{r-1} \Delta_{j+1} \prod_{j=r+2}^{k} (1-\Delta_{j})}_{\eqref{full_dimensional:last2claims_2}/\Delta_0}
      + \underbrace{\binom{k-1}{k-1} \prod_{j=0}^{k-1} \Delta_{j+1}}_{\eqref{full_dimensional:last2claims_1:4}/\Delta_0} \right) \nonumber \\
        = &\Delta_{0} \left( \sum_{r=0}^{k-1} \binom{k-1}{r} \prod_{j=0}^{r-1} \Delta_{j+1} \prod_{j=r+2}^{k} (1-\Delta_{j}) + \sum_{r=1}^{k} \binom{k-1}{r-1} \prod_{j=0}^{r-1}\Delta_{j+1} \prod_{j=r+1}^{k} (1-\Delta_{j}) \right. \nonumber \\
        &\quad\left. - \sum_{r=0}^{k-1} \binom{k-1}{r} \Delta_{r+1} \prod_{j=0}^{r-1} \Delta_{j+1} \prod_{j=r+2}^{k}(1-\Delta_{j}) \right),\label{eq:quasiquasi}
\end{align}
}}
 where in the last equation we used $$\binom{k-1}{0} \prod_{j=2}^{k}(1-\Delta_{j})+\sum_{r=1}^{k-1} \binom{k-1}{r} \prod_{j=0}^{r-1} \Delta_{j+1} \prod_{j=r+2}^{k} (1-\Delta_{j})=\sum_{r=0}^{k-1} \binom{k-1}{r} \prod_{j=0}^{r-1} \Delta_{j+1} \prod_{j=r+2}^{k} (1-\Delta_{j})$$
 and 
$$
\sum_{r=1}^{k-1} \binom{k-1}{r-1} \prod_{j=0}^{r-1} \Delta_{j+1} \prod_{j=r+1}^{k}(1-\Delta_{j}) + \binom{k-1}{k-1} \prod_{j=0}^{k-1} \Delta_{j+1} = \sum_{r=1}^{k} \binom{k-1}{r-1} \prod_{j=0}^{r-1}\Delta_{j+1} \prod_{j=r+1}^{k} (1-\Delta_{j}),
$$
 and 
 \begin{align*}
 - \Delta_{1} \binom{k-1}{0} \prod_{j=2}^{k} (1-\Delta_{j}) -\sum_{r=1}^{k-1} \Delta_{r+1} \binom{k-1}{r} \prod_{j=0}^{r-1} \Delta_{j+1} \prod_{j=r+2}^{k} (1-\Delta_{j})  \\ = - \sum_{r=0}^{k-1} \binom{k-1}{r} \Delta_{r+1} \prod_{j=0}^{r-1} \Delta_{j+1} \prod_{j=r+2}^{k}(1-\Delta_{j}).
\end{align*}
    By doing a change of index variables $r' = r-1$ in the second summand from~\eqref{eq:quasiquasi}, we have that the second and third summands cancel out. Therefore~\eqref{eq:start} equals to
    \begin{align*}
        &\Delta_{0} \left( \sum_{r=0}^{k-1} \binom{k-1}{r} \prod_{j=0}^{r-1} \Delta_{j+1} \prod_{j=r+2}^{k} (1-\Delta_{j}) \right)\\
        &=  
        \Delta_{0} \left( \sum_{r'=1}^{k} \binom{k-1}{r'-1} \prod_{j=0}^{r'-2} \Delta_{j+1} \prod_{j=r'+1}^{k} (1-\Delta_{j}) \right) \\
        &= \Delta_{0} \left( \sum_{r'=1}^{k} \binom{k-1}{r'-1} \prod_{j'=1}^{r'-1} \Delta_{j'} \prod_{j=r'+1}^{k} (1-\Delta_{j}) \right) \\
        &= \Delta_{0} \left( \sum_{r'=0}^{k-1} \binom{k-1}{r'} \prod_{j'=1}^{r'} \Delta_{j'} \prod_{j=r'+2}^{k} (1-\Delta_{j}) \right) \\
        &= \Delta_{0} \omega_{\Delta_{1},\dots,\Delta_{k}}^{k-1}=  \overline \gamma_1 \omega^{k}_{\Delta} ,        
    \end{align*}
where last equation follows by definition~\eqref{linearsystem:gamma_solution2} of $\overline \gamma$.

\smallskip

    \paragraph{Constraints~\eqref{eq:convex_combination_emptyset-new},~\eqref{eq:convex_combination_singletons:integer-new}, and~\eqref{eq:ugly-inequality}.} These constraints can be written compactly as follows:
    \begin{equation}\label{eq:lambda_Qgamma}
        \begin{pmatrix}
            \lambda_{0} \\
            \lambda_{1} \\
            \vdots \\
            \lambda_{k}
        \end{pmatrix} = Q^{k} 
        \begin{pmatrix}
            \gamma_{0} \\
            \gamma_{1} \\
            \vdots \\
            \gamma_{k} 
        \end{pmatrix}
    \end{equation}
    where $Q^{k}$ is the $(k+1) \times (k+1)$ \emph{inverse Pascal upper-triangular matrix} defined as 
    \begin{equation}
    Q_{\ell,r}^{k} := 
    \begin{cases}
        (-1)^{\ell+r} \binom{k-\ell}{r-\ell} & \mbox{if $\ell \leq r$} \\
        0 & \mbox{o.w.}
        \end{cases}
    \end{equation}
    for $\ell,r \in \{0,\dots,k\}$. Indeed, row $\ell$ at the RHS of~\eqref{eq:lambda_Qgamma} corresponds to
    \begin{align*}
        \sum_{r = 0}^{k} Q^{k}_{\ell,r} \gamma_{r} &= \sum_{r = 0}^{k} (-1)^{\ell + r} \binom{k-\ell}{r-\ell} \gamma_{r} \\
        &= \sum_{r = \ell}^{k-\ell} (-1)^{r} \binom{k-\ell}{r} \gamma_{\ell +r} \\
        &= \lambda_{\ell}.
    \end{align*}
    where the last equality follows by~\eqref{eq:ugly-inequality}. We make use of the following lemma.
    \begin{proposition}\label{proposition:Pascal_matrix}
        The matrix $Q^{k}$ is invertible and its inverse is given by the matrix $P^{k}$ defined as
    \begin{equation}
    P_{\ell,r}^{k} =
    \begin{cases}
        \binom{k-\ell}{r-\ell} & \mbox{if $\ell \leq r$} \\
        0 & \mbox{o.w.}
    \end{cases}
    \end{equation}
    for $\ell,r \in \{0, \dots, k\}$. 
    \end{proposition}
    \begin{proof}
    This fact is implicit in the literature, but we give a proof here for completeness. Let $\ell,r \in \{0,\dots,k\}$. We show that
    \begin{equation*}
    (P^{k} Q^{k})_{\ell,r} = 
    \begin{cases}
    1 & \mbox{if $\ell = r$,} \\
    0 & \mbox{if $\ell \neq r$.}
    \end{cases}
    \end{equation*}
First, notice that the $(\ell,r)$ entry of $P^{k}Q^{k}$ is given by
\begin{align*}\label{proposition:proof_Pascal_matrix:eq}
    \sum_{j=0}^{k} P^{k}_{\ell,j} Q^{k}_{j,r} &= \sum_{j=0}^{k}  \binom{k-\ell}{j-\ell} (-1)^{j+r}\binom{k-j}{r-j}   \\
    &= (-1)^{r} \sum_{j=\ell}^{r}  (-1)^{j} \binom{k-\ell}{j-\ell} \binom{k-j}{r-j} \\
    &= (-1)^{r} \sum_{j=0}^{r-\ell} (-1)^{j +\ell} \binom{k-\ell}{j} \binom{(k-\ell)-j}{(r-\ell) - j} \\
    &= (-1)^{r + \ell} \sum_{j=0}^{r-\ell}(-1)^{j} \binom{k-\ell}{r-\ell} \binom{r-\ell}{j} \\
    &= (-1)^{r+\ell} \binom{k-\ell}{r-\ell} \sum_{j=0}^{r-\ell} \binom{r-\ell}{j}
\end{align*}
where the next to last equality follows by the binomial identity $\binom{a}{b} \binom{b}{c} = \binom{a}{c} \binom{a-c}{b-c}$ for integers $a,b,c$, see e.g.~\cite{graham+etal94}. If $\ell = r$, then we have that $(P^{k} Q^{k})_{\ell,r}$ equals $(-1)^{\ell + \ell} \binom{k-\ell}{0} \binom{0}{0} = 1$. On the other hand, if $\ell < r$, then the identity~\eqref{full_dimensional:bin_identity2} implies that $(P^{k} Q^{k})_{\ell,r}$ equals $0$, and if $\ell > r$, the same holds since $\binom{k-\ell}{r-\ell} = 0$.
    \end{proof}
    \noindent The latter proposition and~\eqref{eq:lambda_Qgamma} imply that there is a biunivocal relation between $(\lambda_{0},\dots,\lambda_{k})$ and $(\gamma_{0},\dots,\gamma_{k})$. Therefore, the following system is equivalent to~\eqref{eq:lambda_Qgamma}:
    \begin{equation}\label{eq:gamma_Plambda}
        \begin{pmatrix}
            \gamma_{0} \\
            \gamma_{1} \\
            \vdots \\
            \gamma_{k} 
        \end{pmatrix}
        = P^{k} 
        \begin{pmatrix}
            \lambda_{0} \\
            \lambda_{1} \\
            \vdots \\
            \lambda_{k}
        \end{pmatrix} 
    \end{equation}
    \noindent Next, we verify that $(\overline\gamma_{0},\overline\gamma_{1},\dots,\overline\gamma_{k})$ and $(\overline{\lambda}_{0},\overline{\lambda}_{1},\dots,\overline{\lambda}_{k})$ defined as in~\eqref{linearsystem:gamma_solution2}  and~\eqref{linearsystem:lambda_solution} satisfy~\eqref{eq:gamma_Plambda}. Indeed, for row $\ell \in \{0,\dots,k\}$ we have
    \begin{align*}
        &\overline\gamma_{\ell} = \sum_{r=0}^{k} P^{k}_{\ell,r} \overline\lambda_{r} = \sum_{r=\ell}^{k} \binom{k-\ell}{r-\ell} \overline\lambda_{r} \\
        \iff &\left( \prod_{j = 0}^{\ell-1} \Delta_{j} \right) \frac{\omega^{k-\ell}_{\Delta_{\ell},\dots,\Delta_{k}}}{\omega^{k}_{\Delta}} = \sum_{r=\ell}^{k}  \binom{k-\ell}{r-\ell} \frac{1}{\omega^{k}_{\Delta}} \prod_{j = 0}^{r-1} \Delta_{j} \prod_{j = r+1}^{k} (1 - \Delta_{j})  \\
        \iff &\left( \prod_{j = 0}^{\ell-1} \Delta_{j} \right) \frac{\omega^{k-\ell}_{\Delta_{\ell},\dots,\Delta_{k}}}{\omega^{k}_{\Delta}} = \left(\frac{\prod_{j=0}^{\ell-1} \Delta_{j}}{\omega^{k}_{\Delta}} \right) \sum_{r=\ell}^{k}  \binom{k-\ell}{r-\ell} \prod_{j = \ell}^{r-1} \Delta_{j} \prod_{j = r+1}^{k} (1 - \Delta_{j})  \\
        \iff &\omega^{k-\ell}_{\Delta_{\ell},\dots,\Delta_{k}} = \sum_{r=\ell}^{k}  \binom{k-\ell}{r-\ell} \prod_{j = \ell}^{r-1} \Delta_{j} \prod_{j = r+1}^{k} (1 - \Delta_{j})
    \end{align*}
    where the last statement follows by Definition~\eqref{linearsystem:Dk_definition}. In consequence, \eqref{eq:gamma_Plambda} is satisfied by $(\overline\gamma_{0},\overline\gamma_{1},\dots,\overline{\gamma}_{k})$ and $(\overline{\lambda}_{0},\overline{\lambda}_{1},\dots,\overline{\lambda}_{k})$. This  verifies~\eqref{eq:convex_combination_emptyset-new} and ~\eqref{eq:convex_combination_singletons:integer-new} by looking at rows $\ell = 0$ and $\ell = 1$ in~\eqref{eq:gamma_Plambda}, respectively. At last,~\eqref{eq:ugly-inequality} is verified by \eqref{eq:lambda_Qgamma}.

\subsection{\texorpdfstring{Proofs from \ref{mainthm-d} $\Rightarrow$ \ref{mainthm-a}}{Proofs from \ref{mainthm-d} ⇒ \ref{mainthm-a}}}
\label{app:(D)=>(A)}

We start with an auxiliary lemma. 

\begin{lemma}\label{lemma:D_implies_A_1}
Let $j \in [k]$ and $i \in [n]$. Then 
{\small{
\begin{equation}\label{lemma:integrality-preserving}
    \LS_0^{j}(P) \cap \{ x \in \R^{n} \, : \, x_{i} \in \{0,1\} \} = \LS_0 \left( \LS_0^{j-1}(P) \cap 
    \{ x \in \R^{n} \, : \, x_{i} \in \{0,1\} \}
    \right).
\end{equation}
}}Moreover, if $x \in \LS_0^{j}(P)$ and $x_{i} \in (0,1)$, then there exists $x^{i,1},x^{i,0} \in \LS^{j-1}_{0}(P)$ such that $x^{i,1}_{i} = 1$ and $x^{i,0} = 0$. 
\end{lemma}
\begin{proof}
    See, e.g.,~\cite[Lemma~2.2]{cook+dash01} for a proof of~\eqref{lemma:integrality-preserving}. 
    The second statement follows by the definition of membership in $\LS^{j}_{0}(P)$.
\end{proof}

\begin{proof}[Proof of Lemma~\ref{lemma:D_implies_A_2}]
    By Lemma~\ref{lemma:D_implies_A_1} we have that there exists $\tilde{x}^{i,1}, \tilde{x}^{i,0} \in \LS^{k-j}(P)$ such that a) holds. Next, since $P$ is $(k+1)$-transitive, there exists a $(k+1)$-transitive group $G$ under which $P$ is $G$-invariant. By Proposition~\ref{proposition:symmetry_preserving} the polytope $\LS^{k-j}_{0}(P)$ is $G$-invariant as well. This implies that the pointwise stabilizer of $J$, $G_{J}$, is transitive on the set of indices $[n] \setminus J$, by Proposition~\ref{proposition:ktransitive}. Therefore, we have that $x^{i,1} := \beta_{G_{J}}(\tilde{x}^{i,1})$ and $x^{i,0} := \beta_{G_{J}}(\tilde{x}^{i,1})$ belong to $\LS^{k-j}_{0}(P)$ by the first statement of Proposition~\ref{proposition:fixG}. Since $G_J$ fixes the indices  in $J$, we deduce that a) holds for $x^{i,1}, x^{i,0}$. The second statement of Proposition~\ref{proposition:fixG} implies that $x^{i,1}, x^{i,0}$ also satisfy b).
\end{proof}

\end{document}